\documentclass[a4paper, 11pt]{article}
\usepackage[utf8]{inputenc}
\usepackage{lmodern}
\usepackage[a4paper,left=0.8in,right=0.8in,top=1.25in,bottom=1.25in]{geometry}
\usepackage{boxedminipage}
\usepackage{amsfonts}
\usepackage{amsmath}
\usepackage{amssymb}
\usepackage{graphicx}
\usepackage{amsthm}
\usepackage{subfig}
\usepackage[dvipsnames,svgnames,table]{xcolor}
\usepackage{pgf,tikz,tikz-layers}
\usetikzlibrary{arrows.meta}
\usetikzlibrary{calc}
\usetikzlibrary{angles,quotes}
\usetikzlibrary{positioning}
\usetikzlibrary{intersections}
\usetikzlibrary{backgrounds}
\usepackage{setspace}
\usepackage[colorlinks=true,linkcolor=RoyalBlue,urlcolor=RoyalBlue,citecolor=PineGreen]{hyperref}
\usepackage{enumerate}
\usepackage[nameinlink]{cleveref}
\usepackage{booktabs}
\usepackage{verbatim}
\usepackage{multirow}
\usepackage{paralist}
\usepackage{titlesec}
\usepackage{nicematrix}

\newtheorem{theorem}{Theorem}[section]
\newtheorem{lemma}[theorem]{Lemma}
\newtheorem{conjecture}[theorem]{Conjecture}
\newtheorem{proposition}[theorem]{Proposition}
\newtheorem{corollary}[theorem]{Corollary}

\theoremstyle{definition}\newtheorem{definition}[theorem]{Definition}
\theoremstyle{definition}\newtheorem{example}[theorem]{Example}
\theoremstyle{definition}\newtheorem{remark}[theorem]{Remark}

\colorlet{colbg}{white}
\colorlet{colfg}{black}
\colorlet{colgraphv}{colfg!75!white}
\colorlet{colgraphe}{colfg!55!white}
\colorlet{colG}{DarkSeaGreen}
\definecolor{colR}{HTML}{CC6677}
\definecolor{colO}{HTML}{DDCC77}
\definecolor{colB}{HTML}{6699CC}
\colorlet{colY}{Gold!90!black}

\tikzstyle{vertex}=[fill=colgraphv,circle,inner sep=0pt, minimum size=4pt]
\tikzstyle{edge}=[line width=1.5pt,colgraphe]
\tikzstyle{labelsty}=[font=\scriptsize]
\tikzstyle{angles}=[draw,black!50!white,fill=black!10!white,thick]
\tikzstyle{hline}=[line width=0.5pt,draw=black!30!white,densely dashed]

\newcommand{\notimplies}{\nRightarrow}
\newcommand{\RR}{\mathbb{R}}
\newcommand{\CC}{\mathbb{C}}

\newcommand{\fieldk}{\mathbb{K}}
\newcommand{\ci}{\mathrm{i}}

\newcommand{\cmap}{\mathbf{c}}

\newcommand{\cimg}{C}

\DeclareMathOperator{\rank}{rank}
\DeclareMathOperator{\im}{im}
\DeclareMathOperator{\coker}{coker}

\begin{document}

\title{Angular Constraints on Planar Frameworks}
\author{Sean Dewar, \and Georg Grasegger, \and Anthony Nixon, \and Zvi Rosen, \and William Sims, \and Meera Sitharam,\and David Urizar}
\date{}

\maketitle

\begin{abstract}
    Consider a collection of points in the plane and the sets of slopes or directions of the lines between pairs of points. It is known that the algebraic matroid on the set of direction constraints between the points is equivalent to the algebraic matroid on the set of distances between the points. This is the well-studied generic 2-dimensional rigidity matroid of a graph. This article studies a higher-level construction built on the slope data: an angle constraint system obtained by prescribing relationships between pairs of slopes. The central question we analyze is: when is an angle system rigid, in the sense that every nontrivial motion alters one of the fixed angles?
    
    We formulate the problem in matricial terms for certain edge-colored graphs, finding precise necessary conditions for when such edge-colored graphs are rigid, and a combinatorial characterization of generic rigidity for a special case. We also prove the validity of an equivalent formulation of the angle matroid as the algebraic matroid of a field extension.
\end{abstract}

\section{Introduction}
Consider a collection of points $p = (p_v)_{v \in V})$ indexed by some finite set $V$ 
and the set $d_{vw} = ||p_v - p_w||^2$ of pairwise squared distances between them.
For generic points $p_v$, (e.\,g.\ with algebraically independent coordinates), the algebraic matroid on the set of distances $d_{vw}$
was characterized by Pollaczek-Geiringer \cite{pollaczekgeiringer1927}
and later rediscovered by Laman \cite{laman1970}. This matroid, called the (generic 2-dimensional) \emph{rigidity matroid}, is denoted
$\mathcal{R}_2$ throughout this paper. Viewing the points and squared distances as vertices and edge lengths of a complete graph on vertex set $V$, a set $E$ of edges is
independent in $\mathcal{R}_2$ if and only if it is $(2,3)$-sparse; that is, the subgraph induced by any non-empty
$E' \subset E$ satisfies the inequality
$|E'| \leq 2|V(E')| - 3$. The pair $(G = (V,E),p)$ is called a \emph{bar-joint framework}.

A natural related question examines the sets of slopes or directions of the lines
between pairs of points. In particular, the equations
$m_{vw} = (y_v - y_w)/(x_v-x_w)$,
where $p_v = (x_v,y_v)$ for each $v \in V$. Surprisingly, the algebraic matroid on this set of elements is precisely the rigidity matroid $\mathcal{R}_2$. This was proved by Whiteley
\cite[Proposition AB.14]{whiteley1987parallel} using an analysis of the rigidity matrix. It was reproved later by
Martin using techniques from algebraic geometry \cite{martin2003}.

In this paper, we examine a construction on top of the slope matroid, given by an angle constraint system. Instead of fixing
 slopes of edges, we fix the angles between chosen pairs of edges.
Consider the example in \Cref{fig:first_example}.
\begin{figure}[ht]
    \centering
    \begin{tikzpicture}[scale=3]
        \node[vertex,label={[labelsty]-90:$v_1$}] (1) at (0,0) {};
        \node[vertex,label={[labelsty]-90:$v_3$}] (2) at (1,0) {};
        \node[vertex,label={[labelsty]90:$v_2$}] (3) at (0,1) {};
        \node[vertex,label={[labelsty]90:$v_4$}] (4) at (1,1) {};
        \begin{scope}[on background layer]
            \draw pic (a1) ["",angles,angle radius=1cm] {angle = 1--3--2};
            
            \draw pic (a2) ["",angles,angle radius=1cm] {angle = 4--2--3};
            
            \draw pic (a3) ["",angles,angle radius=1cm] {angle = 3--4--1};

            \draw pic (a4) ["",angles,angle radius=1cm] {angle = 2--1--4};
        \end{scope}
        \draw[edge,colR] (1)edge(3) (3)edge(2) (2)edge(4);
        \draw[edge,colB] (3)edge(4) (4)edge(1) (1)edge(2);
    \end{tikzpicture}
    \quad
    \begin{tikzpicture}[scale=3]
        \begin{scope}[on above layer]
            \node[vertex,label={[labelsty]-90:$v_1$}] (1) at (0,0) {};
            \node[vertex,label={[labelsty]90:$v_2$}] (2) at (0,1) {};
            \node (v41) at (0.4,1.25) {};
            \node (v42) at (0.6,1.25) {};
            \node (v43) at (0.8,1.25) {};
            \node (v44) at (1,1.25) {};
            \node (3) at (1,0) {};
        \end{scope}
        
        \draw[name path = 23, line width = 0.5mm, colR]  (0,1) -- (1.25,-0.25);
        \draw[name path = l4, line width = 0.5mm,colR] (1,-0.25) -- (1,1.25);
        \draw[name path = l3, line width = 0.5mm,colR] (0.8,-0.25) -- (0.8,1.25);
        \draw[name path = l2, line width = 0.5mm,colR] (0.6,-0.25) -- (0.6,1.25);
        \draw[name path = l1, line width = 0.5mm,colR] (0.4,-0.25) -- (0.4,1.25);

        \begin{scope}[on above layer]
            \path[name intersections={of=23 and l1,by=v1}];
            \path[name intersections={of=23 and l2,by=v2}];
            \path[name intersections={of=23 and l3,by=v3}];
            \path[name intersections={of=23 and l4,by=v4}];
        \end{scope}

        \begin{scope}[on background layer]
            \draw pic ["",angles,angle radius=0.8cm] {angle = 1--2--3};

            \draw pic ["",angles,angle radius=0.8cm] {angle = v41--v1--2};
            \draw pic ["",angles,angle radius=0.8cm] {angle = v42--v2--2};
            \draw pic ["",angles,angle radius=0.8cm] {angle = v43--v3--2};
            \draw pic ["",angles,angle radius=0.8cm] {angle = v44--v4--2};
            
        \end{scope}
        \draw[edge,colR] (1)edge(2);
    \end{tikzpicture}
        \quad
    \begin{tikzpicture}[scale=3]
        \begin{scope}[on above layer]
            \node[vertex,label={[labelsty]-90:$v_1$}] (1) at (0,0) {};
            \node[vertex,label={[labelsty]90:$v_2$}] (2) at (0,1) {};
            \node (4) at (0.8,0) {};
        \end{scope}
        
        \draw[name path = 23, line width = 0.5mm, colR]  (0,1) -- (1.25,-0.25);
        \draw[name path = l4, line width = 0.5mm,colR] (1,-0.25) -- (1,1.25);
        \draw[name path = l3, line width = 0.5mm,colR] (0.8,-0.25) -- (0.8,1.25);
        \draw[name path = l2, line width = 0.5mm,colR] (0.6,-0.25) -- (0.6,1.25);
        \draw[name path = l1, line width = 0.5mm,colR] (0.4,-0.25) -- (0.4,1.25);

        \draw[name path = 14, line width = 0.5mm, colB]  (0,0) -- (0.8,0);

        \draw[name path = c, color = gray, dashed] (0.3,0.5) circle (0.58);

        \begin{scope}[on above layer]
            \path[name intersections={of=c and l1,by=31}];
            \path[name intersections={of=c and l2,by=32}];
            \path[name intersections={of=c and l3,by=33}];
        \end{scope}

        \begin{scope}[on background layer]
            \draw pic ["",angles,angle radius=0.55cm] {angle = 2--31--1};
            \draw pic ["",angles,angle radius=0.55cm] {angle = 2--32--1};
            \draw pic ["",angles,angle radius=0.55cm] {angle = 2--33--1};

            \draw pic ["",angles,angle radius=0.8cm] {angle = 4--1--33};
            \draw pic ["",angles,angle radius=1cm] {angle = 4--1--32};
        \end{scope}
        \draw[edge,colR] (1)edge(2);
        \draw[edge,colB] (2)edge(31) (2)edge(32) (2)edge(33) (31)edge(1) (32)edge(1) (33)edge(1);
    \end{tikzpicture}
    \caption{$K_4$ is dependent in $\mathcal R_2$, i.\,e.\ the distance/slope of one of the pairs of points (vertices) is dependent on  the distances/slopes of the remaining pairs, however, the shown set of angles is independent.}
    \label{fig:first_example}
\end{figure}
Noting that for any subset of $m$ edges, fixing $m-1$ pairwise angles determines all pairwise angles, such an angle constraint system partitions the edges into sets such that the pairwise angles are fixed within each set.

The corresponding matroid governing
 angles appears to be a more stubborn object to characterize.
In this paper,
we approach this problem from algebro-geometric and combinatorial perspectives.
Our main results include \Cref{thm:lineangdirection} which gives a matrix-based characterization of angle rigidity, \Cref{thm:2col} which derives a complete combinatorial characterization in a rich special case, and \Cref{cor:algmat} and \Cref{thm:algmat} which provide algebraic matroidal re-formulations.

There are several potential applications of analyzing rigidity for angle frameworks. The multi-agent formation literature uses angle constraint systems to guide multiple agents toward a desired formation based on angle
information \cite{zhao2015bearing, jing2019angle, chen2020angle}. Lacking a combinatorial analysis of angle systems they use combinatorial characterizations arising in the bar-joint setting as a weak approximation.
Combinatorial characterizations as in \Cref{thm:2col} are important for computer-aided design software that computes a solution, i.e., a framework, to a user-specified geometric constraint system, such as those studied by Haller et al.\ in \cite{haller2012}, which include pairs of angle-constrained lines. Such constraint systems typically take doubly exponential time in the number of variables to solve directly via computational algebra.

\subsection{Previous work on angles}
 
Mathematically, the problem has been primarily discussed in two bodies of literature. In the rigidity theory literature, especially the work of Walter Whiteley, angle constraints are considered in the context of direction and length constraints. More recently, various engineering groups have studied this problem as part of multi-agent formation control.

First, we summarize the rigidity theory results. In \cite{servatius1999constraining}, Servatius and Whiteley proved the necessary sparsity bound $|A| \leq 2|V(E)| - 4$, as well as the fact that this sparsity condition is not sufficient to guarantee independence.
Polygon constraints are cited as one demonstration of the failure of sufficiency, but they add that even if polygons are accounted for, the sparsity bound is not sufficient.
In a subsequent paper \cite{eren2003sensor}, Whiteley and co-authors stated further results on angle arrangements. They introduced the \emph{first-order angle matrix} as an analogue to the traditional rigidity matrix. They conjectured that there is no polynomial-time algorithm to check independence of angle arrangements.
They also described Henneberg-type moves to extend angle arrangements: the 0-extension adds one new vertex $v$ to a graph as well as two angle constraints both centered at vertex $v$. The 1-extension adds a vertex $v$, deletes one angle constraint, and adds three new angle constraints. The theorems are stated there unproven, with reference made to an unpublished article entitled ``Constraining plane geometric configurations in CAD: Angle''. They conclude that these two types of extensions were insufficient to construct all angle-rigid arrangements, since they do not produce vertices of degree five.

In the 2006 doctoral thesis of Zhou \cite{zhou2006combinatorial}, the author considers the problem of angle constraints. A rigorous proof of the validity of 0-extensions (which the author calls ``gradual construction'') is presented, as are some necessary combinatorial conditions for rigidity. Other studies of angles have been made but usually with some fundamental twist, e.\,g.\ considering circle arrangements with constrained angles of intersection as in \cite{saliola2004constraining}.

The multi-agent formation literature is differentiated from the rigidity literature
both in its goals and its mathematical approach. In particular, they often feature practical
results that can be used to guide multiple agents toward a desired formation based on angle
information. The mathematical tools tend to come from control theory and analysis. For
example, the 2015 paper \cite{zhao2015bearing} begins from the 2-d rigidity-theoretic work and generalizes to higher dimensions. Then they use Lyapunov methods to define a control law that can stabilize bearing-rigid formations.
Other recent papers \cite{jing2019angle, chen2020angle} stay focused on the 2-d setting,
but further explore the control law and how perturbed formations stabilize under a control law.
Given the different objectives and toolkit, we leave further integration of the literature for future work.

\subsection{Analogous cases in the literature}

At first glance, the angle matroid appears to be a special case of the point-line incidence
structures studied by Jackson and Owen in \cite{jackson2016characterisation}. In that article, the authors
characterize structures with two types of objects---points and lines---and three types of relations:
\begin{inparaenum}[(i)]
    \item angles between pairs of lines;
    \item perpendicular distance from a point to a line;
    \item distance from point to point.
\end{inparaenum}
Assuming these quantities are generic, they fully characterize the resulting matroid. If you set the distance between $p$ and $\ell$ to $0$ at $p$ lies on $\ell$ and the resulting setup is exactly of the type we consider.
Unfortunately, once genericity is broken, the characterization of  \cite{jackson2016characterisation} no longer applies.

Another seemingly similar setup is that of frameworks with coordinated edge motions, studied by Schulze, Serocold and Theran in \cite{schulze2022coordinated}. In their setting, subcollections of edges
are assigned to ``coordinated classes.'' In addition to the standard rigid motions,
the edges of each color are allowed to change length (additively) by the same fixed amount while still being considered equivalent.
They characterize the resulting matroid (in the $d=2$ case) as the matroid union of the
rigidity matroid $\mathcal{R}_2$ with the transversal
matroid whose bases take one
element from each of the coordinated classes.
We believe that the matroid here takes a very similar form, though our characterization is incomplete.
Indeed, they note in Section 5.3 of their paper, ``the rigidity analysis of such coordinated frameworks seems more complex than the one considered in this paper.''

\subsection{Outline}
The structure of the paper is as follows: In \Cref{s:basic}, we 
formally define angular analogues of concepts from rigidity theory.
In \Cref{sec:necessary}, we prove an important property of angle-rigid frameworks and conjecture that the property is also sufficient for 
angle-rigidity. In \Cref{sec:extensions}, we define extension moves on angle frameworks, much like Henneberg moves, that allow us to characterize an important subclass of arrangements. Finally, in \Cref{sec:algmatroid}, we derive an 
algebraic matroid formalizing angle-rigidity, much as we have in standard 
rigidity theory. Taken as a whole, this represents an important step in
developing angular constraints as an analog to classical rigidity theory.

\section{Basic definitions}
\label{s:basic}
In this section we develop the general theory of angle-rigid frameworks,
from the perspective of the geometry of realizations.

\begin{definition}
    Let $G = (V,E)$ be a graph. We call an injective map $p: V\to \RR^2$ a \emph{realization} of $G$, and we denote by $R^G \subset (\RR^2)^V$ the (Zariski-open) \emph{realization space} of $G$.

    \begin{itemize}
        \item An \emph{angle index set} $A$ is a set of distinct unordered pairs in $E$; for example,\\ $A = \{\{(a_1 b_1),(c_1 d_1)\},\ldots, \{(a_k b_k),(c_k d_k)\} \mid  a_i b_i,c_id_i \in E\}$.
        \item The triple $(G,A,p)$ is called an \emph{angle framework}.
        \item The \emph{angle map} is the map sending each realization to the resulting angles indexed by $A$:\\
        $\theta_A: R^G \rightarrow \RR^{A}, ~ \mathbf{p} = (p_v)_{v \in V} \mapsto \left(\arccos \left( \frac{(p_a-p_b)\cdot (p_c-p_d)}{\|p_a-p_b\| \|p_c-p_d\|} \right) \right)_{\{ab,cd\} \in A}$.
        \item The \emph{angle graph} $\tilde{G}(A)$ is the graph with vertices indexed
        by $E$ and edges by $A$.
        \item The \emph{edge support} $E(A)$ denotes the set of edges in $G$ appearing as half of some pair in the angle index set.
    \end{itemize}
\end{definition}

\begin{example}
    Let us consider the graph $G=(V,E)$ with vertex set $V=\{v_1,v_2,v_3,v_4\}$ and edges $E=\{(v_1v_2),(v_1v_3),(v_1v_4),(v_2v_3),(v_2v_4)\}$.
    Let further $A=\{\{(v_1v_2),(v_1v_3)\},\{(v_2v_3),(v_2v_4)\}\}$.
    \Cref{fig:angleframe} (left) shows the angle framework $(G,A,p)$, where $A$ is indicated by colors and $p_{v_1}=(0,0)$, $p_{v_2}=(1,0)$, $p_{v_3}=(1,1)$ and $p_{v_4}=(0,1)$.
    The angle map $\theta_A$ is then defined by $\theta_A(\{(v_1v_2),(v_1v_3)\})=\theta_A(\{(v_2v_3),(v_2v_4)\})=\pi/4$.
    The angle graph $\tilde{G}(A)$ is depicted on the right.
    \begin{figure}[ht]
        \centering
        \begin{tikzpicture}[scale=1.25]
            \node[vertex,label={[labelsty]below:$v_1$}] (v1) at (0,0) {};
            \node[vertex,label={[labelsty]below:$v_2$}] (v2) at (2,0) {};
            \node[vertex,label={[labelsty]above:$v_3$}] (v3) at (2,2) {};
            \node[vertex,label={[labelsty]above:$v_4$}] (v4) at (0,2) {};
            \draw[edge,colB] (v1)to(v2) (v1)to(v3);
            \draw[edge,colR] (v2)to(v3) (v2)to(v4);
            \draw[edge] (v1)to(v4);
        \end{tikzpicture}
        \qquad\qquad
        \begin{tikzpicture}
            \foreach \a [count=\i from 0] in {{$(v_1v_2)$},{$(v_1v_3)$},{$(v_1v_4)$},{$(v_2v_3)$},{$(v_2v_4)$}}
            {
                \node[vertex,label={[labelsty]\i*72+18:\a}] (e\i) at (\i*72+18:1.5) {};
            }
            \draw[edge,colB] (e0)--(e1);
            \draw[edge,colR] (e3)--(e4);
        \end{tikzpicture}
        \caption{An angle framework (left) with $A$ indicated in colors, and the corresponding angle graph (right).}
        \label{fig:angleframe}
    \end{figure}
\end{example}

\noindent We would like to drop some of the redundancy in the definition of angle frameworks. To do this, we require notions of equivalence up to
some transformation:
\begin{definition}
    Let $G$ be a graph and $A$ an angle index set.  Let $(G,A,p)$ and $(G,A,q)$ be
    two angle frameworks.
    \begin{itemize}
        \item $(G,p)$ and $(G,q)$ are called \emph{similar} if there exists a linear isometry $T$, a vector $z \in \RR^2$ and a scalar $\lambda>0$ such that $q_v = \lambda T (p_v) +z$ for each $v \in V$;
        the affine map $x \mapsto \lambda T(x) + z$ is called a \emph{similarity}.
        \item $(G,A,p)$ and $(G,A,q)$ are \emph{equivalent} if $\theta_A(p) = \theta_A(q)$.
    \end{itemize}
\end{definition}

Let $A_1$ and $A_2$ be angle index sets of $G$.  
Observe that if $\tilde{G}(A_1)$ and $\tilde{G}(A_2)$ partition $E$ into the same set of connected components $c_1,\ldots, c_k$, then  $(G,A_1,p) $ is equivalent to $(G,A_1,q)$ if and only if $(G,A_2,p)$ is equivalent to $(G,A_2,q)$.  
Hence, equivalence is completely determined by the connected components of $\tilde{G}(A)$, which can be viewed in terms of edge-colored graphs.

\begin{definition}
    Let $G=(V,E)$ be a graph.
    We define the \emph{color map} $\cmap$, associating a color to every edge, and call
    $(G,\cmap)$ an \emph{edge-colored graph}.
    By abuse of notation justified above, we also call $(G,\cmap,p)$ an angle framework.
    Furthermore, we let $G_c=(V,E_c)$ where $E_c$ is the set of all edges of color $c$, and we denote by $\cimg_G=\cmap(E)$ and omit $G$ in case it is clear.
\end{definition}

\subsection{The angle-rigidity matrix}

Let $G=(V,E)$ be a graph and $(G,\cmap,p)$ an angle framework.  
For every point $p_v = (x_v,y_v) \in \RR^2$,
$p_v^\perp$ denotes the point formed by applying a $90^\circ$ counter-clockwise rotation to $p_v$.
As a short-hand, we use the notation $p^\perp$ for $(p_v^\perp)_{v \in V}$. For any two vertices $v,w \in V$, we set $X_{vw} = x_v - x_w$, $Y_{vw} = y_v - y_w$, and the vector $P_{vw} = p_v - p_w = (X_{vw},Y_{vw})$.

By differentiating the angle constraints $\theta_A(p_v) = constant$ and scaling the result, we observe that a map $u:V \rightarrow \RR^2$ is an infinitesimal deformation of $(G,\cmap,p)$ that preserves angles between pairs of edges of the same color (now referred to as an \emph{infinitesimal flex}) if and only if for every pair $ab,vw \in E$ with $\cmap(ab)=\cmap(vw)$,
we have
\[ 
    \left( \frac{(P_{ab}\cdot P_{vw})P_{ab} - (P_{ab}\cdot P_{ab})P_{vw}}{P_{ab}\cdot P_{ab}} \right)\cdot (u_a-u_b)
    \: + \: \left( \frac{(P_{ab}\cdot P_{vw})P_{vw} - (P_{vw}\cdot P_{vw})P_{ab}}{P_{vw}\cdot P_{vw}} \right)\cdot (u_v-u_w) = 0.
\]

\noindent By cancellation, we derive that

\[    
    \big( (P_{ab}\cdot P_{vw}) P_{ab} - (P_{ab}\cdot P_{ab})P_{vw} \big)\cdot P_{ab} = 0 
    \quad \text{   and    } \quad
    \big( (P_{ab} \cdot P_{vw})P_{vw} - (P_{vw}\cdot P_{vw})P_{ab}\big)\cdot P_{vw} = 0.
\]

\noindent Hence,

\[
    (P_{ab}\cdot P_{vw}) P_{ab} - (P_{ab}\cdot P_{ab})P_{vw} = \alpha_{ab} P_{ab}^\perp 
    \quad \text{   and   } \quad
    (P_{ab} \cdot P_{vw})P_{vw} - (P_{vw}\cdot P_{vw})P_{ab} = \alpha_{vw} P_{vw}^\perp.\:
\]
for some scalars $\alpha_{ab},\alpha_{vw}\in \RR$.
By measuring the two vectors,
we see that $\alpha_{ab} =  \alpha_{vw} = -P_{ab}^\perp \cdot P_{vw}$.
This allows us to simplify the infinitesimal flex constraint condition to obtain
\begin{align}\label{eq:constraints}
    \left( \frac{P_{ab}^\perp}{P_{ab}\cdot P_{ab}} \right)\cdot (u_a-u_b)
    - \left( \frac{P_{vw}^\perp}{P_{vw}\cdot P_{vw}} \right)\cdot (u_v-u_w) =0.
\end{align}
We say that an infinitesimal flex is \emph{trivial} if it is a restriction of an infinitesimal similarity to the points $\{p_v:v \in V\}$.
It can be easily checked that \cref{eq:constraints} holds for any choice of vertices $a,b,c,d$ with $a\neq b$ and $c \neq d$ when $u$ is a trivial infinitesimal flex.
The set of trivial infinitesimal flexes forms a linear subspace of the linear space of infinitesimal flexes.
Since similarities are formed from translations, rotations and scalings,
the following result is immediate.

\begin{lemma}\label{l:trivflex}
    Let $(G,\cmap,p)$ be an angle framework.
    If there exist vertices $v,w \in V$ where $p_v\neq p_w$,
    then the following vectors form a basis of the trivial infinitesimal flexes of $(G,\cmap,p)$:
    \begin{equation*}
       u^{(1,0)} =((1,0))_{v \in V}, \qquad u^{(0,1)} =((0,1))_{v \in V}, \qquad
        u^p = (p_v)_{v \in V}, \qquad u^{p^\perp} = (p_v^\perp)_{v \in V}.
    \end{equation*}
\end{lemma}

\begin{lemma}\label{lem:constraint}
    Let $(G,\cmap,p)$ be an angle framework.
    Let $u \in (\RR^2)^V$.
    Then $u$ is an infinitesimal flex of $(G,\cmap,p)$
    if and only if for each color $c\in \cimg$,
    there exists $\lambda_c \in \RR$ so that for each $vw \in E$ with $\cmap(vw)=c$,\\[2mm]
    \begin{align*}
        P_{vw}^\perp \cdot (u_v-u_w) = \lambda_c P_{vw}\cdot P_{vw}.
    \end{align*}
\end{lemma}

\begin{proof}
    For each color $c$ choose an edge $vw \in E$ with $\cmap(vw)=c$.
    We now set
    \begin{align*}
       \lambda_c := \dfrac{P_{vw}^\perp \cdot (u_v-u_w)}{P_{vw}\cdot P_{vw}}.
    \end{align*}
    The result follows by inspection of the constraint system given in \cref{eq:constraints}.
\end{proof}

\begin{definition}
    The \emph{angle-rigidity matrix} is the $|E| \times (2|V|+|\cimg|)$ matrix, defined blockwise as
    \begin{align*}
	   R(G,\cmap,p) := \left[ R(G,p) \quad M(G,\cmap,p) \right],
    \end{align*}
    where $R(G,p)$ is the standard 2-d rigidity matrix, and $M(G,\cmap,p)$ is the $|E| \times |\cimg|$ matrix with entries
    \begin{align*}
    	M(G,\cmap,p)_{vw,c} :=
    	\begin{cases}
    		-P_{vw} \cdot P_{vw} = -X_{vw}^2 - Y_{vw}^2, &\text{if } \cmap(vw)=c,\\
    		0, &\text{otherwise}.
    	\end{cases}
    \end{align*}
\end{definition}

By \Cref{lem:constraint},
a map $u\colon V \rightarrow \RR^2$ is an infinitesimal flex of $(G,\cmap,p)$ if and only if there exists $\lambda : \{1,\ldots,|\cimg|\} \rightarrow \RR$ such that $(u,\lambda) \in \ker R(G,c,p^\perp)$.
From this we can immediately deduce the following lemma.

\begin{lemma}\label{l:trivflex2}
    Let $(G,\cmap,p)$ be an angle framework.
    If there exist vertices $v,w \in V$ where $p_v\neq p_w$,
    then the following vectors are contained in the kernel of $R(G,\cmap,p)$:
    \begin{equation}\label{eq:trivialstresses}
        (u^{(1,0)}, \mathbf{0}), \qquad (u^{(0,1)}, \mathbf{0}), \qquad
        (u^{p^\perp}, \mathbf{0}), \qquad (u^{p}, \mathbf{1}),
    \end{equation}
    where each of the $u$ vectors is defined as in \Cref{l:trivflex}, and $\mathbf{0},\mathbf{1} : \{1,\ldots,|\cimg|\} \rightarrow \RR$ are the constant maps $i \mapsto 0$ and $i \mapsto 1$ respectively.
\end{lemma}

\begin{remark}\label{rem:triv}
    It follows from \Cref{l:trivflex2} that for $(G,\cmap,p)$ with vertices $v,w \in V$ such that $p_v\neq p_w$, the rank of the matrix $R(G,\cmap,p)$ is the same as the rank of the matrix formed by deleting the columns corresponding to the vertices $v$ and $w$.
\end{remark}

\begin{example}
    Let us consider the triangle graph where all edges have the same color and the same graph where there are edges in two colors (see \Cref{fig:triangles}).
    \begin{figure}[ht]
        \centering
        \begin{tikzpicture}[scale=2]
            \node[vertex,label={[labelsty]below:$v$}] (a) at (0,0) {};
            \node[vertex,label={[labelsty]below:$w$}] (b) at (1,0) {};
            \node[vertex,label={[labelsty]above:$z$}] (c) at (60:1) {};
            \draw[edge] (a)to node[below,labelsty] {c} (b) (b)to node[right,labelsty] {c} (c) (c)to node[left,labelsty] {c} (a);
        \end{tikzpicture}
        \qquad
        \begin{tikzpicture}[scale=2]
            \node[vertex,label={[labelsty]below:$v$}] (a) at (0,0) {};
            \node[vertex,label={[labelsty]below:$w$}] (b) at (1,0) {};
            \node[vertex,label={[labelsty]above:$z$}] (c) at (60:1) {};
            \draw[edge] (a)to node[below,labelsty] {$c_1$} (b) (b)to node[right,labelsty] {$c_1$} (c);
            \draw[edge,colB] (c)to node[left,colB,labelsty] {$c_2$} (a);
        \end{tikzpicture}
        
        \caption{Colored triangles $(K_3,\cmap)$ and $(K_3,\cmap')$.}
        \label{fig:triangles}
    \end{figure}
    Then $R(K_3,c,p)$ and $R(K_3,c',p)$ are given by
    \begin{align*}
        &\begin{pmatrix}
             X_{wv} & Y_{wv} & -X_{wv} & -Y_{wv} & 0 & 0 & -X_{wv}^2 - Y_{wv}^2 \\
             X_{wz} & Y_{wz} & 0 & 0 & -X_{wz} & -Y_{wz} & -X_{wz}^2 - Y_{wz}^2 \\
             0 & 0 & X_{vz} & Y_{vz} & -X_{vz} & -Y_{vz} & -X_{vz}^2 - Y_{vz}^2 \\
        \end{pmatrix}
        \quad
        \text{ \normalsize and }\\
        & \begin{pmatrix}
             X_{wv} & Y_{wv} & -X_{wv} & -Y_{wv} & 0 & 0 & -X_{wv}^2 - Y_{wv}^2 & 0 \\
             X_{wz} & Y_{wz} & 0 & 0 & -X_{wz} & -Y_{wz} & -X_{wz}^2 - Y_{wz}^2 & 0 \\
             0 & 0 & X_{vz} & Y_{vz} & -X_{vz} & -Y_{vz} & 0 & -X_{vz}^2 - Y_{vz}^2 \\
        \end{pmatrix}, \quad \text{respectively.}
    \end{align*}
\end{example}

\subsection{Three types of angle-rigidity}
\label{ss:three}
In this section we define rigidity properties for angle frameworks and colored graphs.
\begin{definition}
    Let $(G,\cmap,p) $ be an angle framework.
    We define three types of angle-rigidity in analogy with standard notions of rigidity,
    see for example, \cite{gss1993,whiteley1996}:
    \begin{enumerate}
        \item $(G,\cmap,p)$ is \emph{locally angle-rigid} if all angle frameworks equivalent and sufficiently close to $(G,\cmap,p)$ are similar to $(G,\cmap,p)$.
        \item $(G,\cmap,p)$ is \emph{globally angle-rigid} if all angle frameworks equivalent to $(G,\cmap,p)$ are similar to $(G,\cmap,p)$. 
        \item $(G,\cmap,p)$ is \emph{infinitesimally angle-rigid} if the null space of the angle-rigidity matrix $R(G,\cmap,p)$ is precisely the space of trivial infinitesimal flexes.
    \end{enumerate}
    Furthermore, $(G,\cmap,p)$ is \emph{minimally locally/infinitesimally angle-rigid} if, in addition to being locally/infinitesimally angle-rigid, deleting any edge (and adjusting the coloring accordingly) produces an angle framework that is not locally/infinitesimally angle-rigid.
\end{definition}

From the previous section, we see that the following holds.

\begin{theorem}\label{thm:lineangdirection}
    Let $(G,\cmap,p)$ be an angle framework with $|V| \geq 2$.
    Then $(G,\cmap,p)$ is infinitesimally angle-rigid if and only if $\rank R(G,\cmap,p)= 2|V| + |\cimg| -4$.
\end{theorem}

\begin{proof}
    The matrix $R(G,\cmap,p)$ is formed from $R(G,c,p^\perp)$ by applying a column reordering and multiplying some columns by $-1$,
    hence $\rank R(G,\cmap,p) = \rank R(G,c,p^\perp)$.
    The result now follows from \Cref{l:trivflex2}.
\end{proof}

\begin{proposition}\label{p:asimowroth}
    For a fixed angle framework $(G,\cmap,p)$, both infinitesimal and global
    angle-rigidity imply local angle-rigidity.
    No other implication among the notions of angle-rigidity holds.
\end{proposition}

\begin{proof}
    \noindent \emph{(Infinitesimal angle-rigidity $\implies$ Local angle-rigidity).}
    We first observe that if $(G,\cmap,p)$ is not \emph{affinely spanning} (i.e., the affine span of $\{p_v:v \in V\}$ is $\RR^2$),
    then $(G,\cmap,p)$ is either a single vertex or two vertices joined by an edge (in which case the result is obvious) or else $(G,\cmap,p)$ is not infinitesimally angle-rigid.
    
    Hence, suppose that $(G,\cmap,p)$ is affinely spanning.
    To ``quotient out'' the trivial motions of $(G,\cmap,p)$,
    we can ``pin'' two vertices $a$ and $b$ which are joined by an edge to the points $p_a$ and $p_b$, respectively (it is easy to see that the framework requires at least one edge to be infinitesimally angle-rigid).
    This corresponds to restricting the domain of the map $\theta_A$ (where $A$ is any set of angles with corresponding edge coloring $c$) to the set of realizations with $a$ at $p_a$ and $b$ at $p_b$.
    Since $(G,\cmap,p)$ is infinitesimally angle-rigid,
    it follows from \Cref{lem:constraint,l:trivflex2} that the Jacobian of our domain-restriction for $\theta_A$ at $p$ has rank $2|V|-4$ (since we have now removed the trivial infinitesimal flexes),
    which is maximal.
    By applying the constant rank theorem to our new map,
    we see that any equivalent framework $(G,c,q)$ with $q_a=p_a$ and $q_b=p_b$ that is sufficiently close to $(G,\cmap,p)$ is exactly $(G,\cmap,p)$, which gives the desired result. \\[2ex]
    \noindent \emph{(Global angle-rigidity $\implies$ Local angle-rigidity).}
    By definition.\\[2ex]
    \noindent \emph{(Infinitesimal angle-rigidity $\notimplies$ Global angle-rigidity and Local angle-rigidity $\notimplies$ Global angle-rigidity).}
    We present the following angle framework which is infinitesimally angle-rigid
    (thus locally angle-rigid) but not globally angle-rigid.
    The graph $G $ is $K_4$, the color map assigns the edges $\{(vw),(vb),(wa),(ab)\}$ to color $c_1$, and $\{(va),(wb)\}$ to color $c_2$,
    and the embedding is given by

    \[
        p_v = (0,0), \quad p_w = (1,0), \quad p_a = \left(1 + \frac{1}{\sqrt{2}}, \frac{1}{\sqrt{6}}\right),
        \quad p_b = \left(\frac{1+\sqrt{2}}{2}, \frac{1+\sqrt{2}}{2\sqrt{3}}\right).
    \]

    \begin{figure}
        \centering
        \begin{tikzpicture}[scale=3]
            \node[vertex,label={[labelsty]-90:$v$}] (1) at (0,0) {};
            \node[vertex,label={[labelsty]-90:$w$}] (2) at (1,0) {};
            \node[vertex,label={[labelsty]0:$a$}] (3) at ($(1,0)+1/sqrt(2)*(1,0)+1/sqrt(6)*(0,1)$) {};
            \node[vertex,label={[labelsty]90:$b$}] (4) at ($(0.5,0)+sqrt(2)/2*(1,0)+1/(2*sqrt(3))*(0,1)+sqrt(2)/(2*sqrt(3))*(0,1)$) {};
            \begin{scope}[on background layer]
                \draw pic (a1) ["",angles,angle radius=1cm] {angle = 2--1--4};
                \node[labelsty,above=5pt] at (a1) {30°};
                \draw pic (a2) ["",angles,angle radius=0.25cm] {angle = 3--2--1};
                \node[labelsty,above left=-1pt] at (a2) {150°};
                \draw pic (a4) ["",angles,angle radius=0.5cm] {angle = 1--4--3};
                \node[labelsty,above right=4pt] at (a4) {120°};
        
                \path[name path=D1] (1)--(3);
                \path[name path=D2] (2)--(4);
                \path[name intersections={of=D1 and D2,by=M}];
                \draw pic (am) ["",angles,angle radius=0.35cm] {angle = 3--M--4};
                \node[labelsty,above right=1pt] at (am) {60°};
            \end{scope}
            \draw[edge,colR] (1)edge(2) (2)edge(3) (3)edge(4) (4)edge(1);
            \draw[edge,colB] (2)edge(4) (1)edge(3);
        \end{tikzpicture}
        \quad
        \begin{tikzpicture}[scale=4]
            \node[vertex,label={[labelsty]-90:$v$}] (1) at (0,0) {};
            \node[vertex,label={[labelsty]-90:$w$}] (2) at (1,0) {};
            \node[vertex,label={[labelsty]90:$a$}] (3) at ($(1,0)-1/sqrt(2)*(1,0)+1/sqrt(6)*(0,1)$) {};
            \node[vertex,label={[labelsty]180:$b$}] (4) at ($(0.5,0)-sqrt(2)/2*(1,0)-1/(2*sqrt(3))*(0,1)+sqrt(2)/(2*sqrt(3))*(0,1)$) {};
            \draw[edge,colR] (1)edge(2) (2)edge(3) (3)edge(4) (4)edge(1);
            \draw[edge,colB] (2)edge(4) (1)edge(3);
    
            \begin{scope}[on background layer]
                \coordinate (h1) at (-0.25,0);
                \coordinate (h2) at (1.25,0);
                \coordinate (h4) at ($(1)!1.75!(4)$);
                
                \draw pic (a1) ["",angles,angle radius=0.5cm] {angle = 4--1--h1};
                \node[labelsty,below left=3pt] at (a1) {30°};
                \draw pic (a2) ["",angles,angle radius=0.3cm] {angle = h2--2--3};
                \node[labelsty,above right=-1pt] at (a2) {150°};
                \draw pic (a4) ["",angles,angle radius=0.3cm] {angle = 3--4--h4};
                \node[labelsty,above=4pt] at (a4) {120°};
        
                \path[name path=D1] (1)--(3);
                \path[name path=D2] (2)--(4);
                \path[name intersections={of=D1 and D2,by=M}];
                \draw pic (am) ["",angles,angle radius=0.35cm] {angle = 2--M--3};
                \node[labelsty,above right=1pt] at (am) {60°};

                \draw[hline] (h1)--(h2);
                \draw[hline] (1)--(h4);
            \end{scope}
        \end{tikzpicture}
        \caption{An angle framework with two different realizations.}
        \label{fig:angles}
    \end{figure}

    Computing the angle-rigidity matrix explicitly demonstrates that $(G,\cmap,p)$
    is infinitesimally angle-rigid. However, there is another realization (\Cref{fig:angles}) with the same angles:

    \[
        p_v = (0,0), \quad p_w = (1,0), \quad p_a =\left(1-\frac{1}{\sqrt{2}},\frac{1}{\sqrt{6}}\right),\quad p_b = \left(\frac{1-\sqrt{2}}{2},\frac{\sqrt{2}-1}{2\sqrt{3}}\right).
    \]

    Direct computation of the angle between 12 and 24 shows that the realizations are dissimilar.  Note that our definition of similarity considers two angles to be the same as long as the pairs of defining
    lines intersect in the same set of four angles.     
    These realizations were obtained by computing the defining system of polynomial equations, and
    solving explicitly for the vertex coordinates after fixing some angles.\\[2ex]
    \noindent \emph{(Global angle-rigidity $\notimplies$ Infinitesimal angle-rigidity and Local angle-rigidity $\notimplies$ Infinitesimal angle-rigidity).}
    We present an angle framework which is globally angle-rigid (thus locally angle-rigid) but not infinitesimally angle-rigid. Again we set $G = K_4$. This time, the edges $\{(vw),(va),(vb),(wa)\}$ have color $c_1$ and
    $\{(wb),(ab)\}$ have $c_2$. The embedding into $\RR^2$ (\Cref{fig:example}) is given by\\[2mm]
        \centerline{$p_v = (0,0), \quad p_w = (1,0), \quad p_a = (2,1), \quad p_b = (0,1).$}\\[2mm]        
    Plugging these values into the angle-rigidity matrix, we see that the null space has something extra that corresponds to the infinitesimal vertical motion of $p_b$.
    This implies that the angle framework is \emph{not}
    infinitesimally angle-rigid.

    \begin{figure}[ht]
        \centering
        \begin{tikzpicture}[scale=1.5]
            \draw[black!40!white] (1,1) circle[radius=1cm];
            \draw[-{Latex[]},black!40!white,dotted] (0,0)--(0,2);
            \node[vertex,label={[labelsty]180:$v$}] (1) at (0,0) {};
            \node[vertex,label={[labelsty]-90:$w$}] (2) at (1,0) {};
            \node[vertex,label={[labelsty]0:$a$}] (3) at (2,1) {};
            \node[vertex,label={[labelsty]180:$b$}] (4) at (0,1) {};
            \draw[edge,colR] (1)edge(2) (1)edge(3) (1)edge(4) (2)edge(3);
            \draw[edge,colB] (2)edge(4) (3)edge(4);
        \end{tikzpicture}
        \caption{An infinitesimal vertical motion of a globally angle-rigid angle framework.}
        \label{fig:example}
    \end{figure}
    
    Any configuration of the framework has a representative with
    $p_v = (0,0)$, $p_w = (1,0)$ up to the group action. Take any embedding
    realizing $(G,\cmap)$ with the same angles. Then, it has a curve of
    representatives keeping $p_v$ and $p_w$
    constant. The angles between $(vw), (va)$, and $(wa)$ would then
    determine the location of $a$ to be $(w,v)$.
    Then $p_b$ must be on the $y$-axis based on the angle $\angle wvb$; it must also
    be on the black circle so that the angle $\angle wba$ is fixed. The circle and
    line intersect in only one point, so the graph is globally angle-rigid.
\end{proof}

Having defined infinitesimal rigidity of angle-frameworks, we 
now connect it to infinitesimal rigidity of underlying bar-joint frameworks.
Fix $\coker M$ to be the cokernel (also known as the left kernel or left null space) of the matrix $M$.
We say that $(G,\cmap,p)$ is \emph{independent} if $\coker R(G,\cmap,p) = \{0\}$.
An element of $\coker R(G,\cmap,p)$ is called an \emph{equilibrium stress} of $(G,\cmap,p)$.
Hence, $(G,\cmap,p)$ is independent if and only if it has no non-zero equilibrium stresses.
The structure of $R(G,\cmap,p)$ implies that $\omega \in \RR^E$ is an equilibrium stress of $(G,\cmap,p)$ if and only if it is an equilibrium stress of $(G,p)$ and for every color~$c$ we have
\begin{align*}
    \sum_{vw \in E_c} \omega_{vw} \|p_v-p_w\|^2=0.
\end{align*}

\begin{proposition}\label{prop:stress}
    Let $(G,\cmap,p)$ be an angle framework. 
    For $c_i \in \cimg$, an equilibrium stress of $(G_{c_i},p)$ induces an equilibrium stress of $(G,\cmap,p)$.
    If $(G,\cmap,p)$ is independent, then each bar-joint framework $(G_{c_j},p)$ is independent for each color $c_j$.
\end{proposition}

\begin{proof}
    Let $(G,\cmap,p)$ be an angle framework and fix a color $c_i$. 
    Let $\omega$ be an equilibrium stress of the bar-joint framework $(G_{c_i},p)$.
    Define $\tilde{\omega} \in \RR^E$ by $\tilde{\omega}_{vw}=\omega_{vw}$ for all $vw \in E_{c_i}$ and $\tilde{\omega}_{vw}=0$ otherwise.
    As $\omega$ is an equilibrium stress of $(G_{c_i},p)$,
    we have that $\sum_{vw \in E_{c_i}} \omega_{vw} \|p_v-p_w\|^2=0$,
    see \cite{connelly1982energy}.
    Thus, $\tilde{\omega}$ lies in the cokernel of $R(G,\cmap,p)$.
    Since any equilibrium stress of $(G_{c_j},p)$ would induce an equilibrium stress
    of $(G,\cmap,p)$, independence of $(G,\cmap,p)$ forces independence of $(G_{c_j},p)$.
\end{proof}

\begin{corollary}\label{cor:mono}
    Let $(G,\cmap,p)$ be a monochromatic angle framework.
    Then $(G,\cmap,p)$ is infinitesimally angle-rigid (resp.\ independent) if and only if the bar-joint framework $(G,p)$ is infinitesimally rigid (resp.\ independent).
\end{corollary}

\begin{proof}
    It follows from \Cref{prop:stress} that $\omega \in \RR^E$ is an equilibrium stress of $(G,\cmap,p)$ if and only if it is an equilibrium stress of $(G,p)$. Hence, $\rank R(G,\cmap,p)=\rank R(G,p)$.
\end{proof}

\subsection{Genericity and angle-rigid graphs}

\Cref{fig:example} gives an example of a locally angle-rigid (in fact, globally angle-rigid) graph which is not
infinitesimally angle-rigid,
however it relies on a very specific geometric coincidence.
In particular, the tangent line to the unique circle containing points $w$, $a$ and $b$
contains the edge $vb$. This coincidence leads us to define an angle-version
of the rigidity-theoretic concept of generic rigidity.

\begin{definition}
    The colored graph $(G,\cmap)$ is said to be \emph{(minimally) angle-rigid} if there exists a non-empty Zariski open subset $S \subset \RR^{2|V|}$ such that for all $p \in S$ the associated angle framework $(G,\cmap,p)$ is (minimally) locally angle-rigid.
\end{definition}

It follows, using \Cref{p:asimowroth}, that our Zariski open subset $S$ can be chosen such that each $p \in S$ is associated to an infinitesimally angle-rigid angle framework $(G,\cmap,p)$.

\begin{proposition}\label{prop:anglerigid}
    Suppose there exists an embedding $p$ such that $(G,\cmap,p)$ is (minimally) infinitesimally angle-rigid. Then the colored graph $(G,\cmap)$ is (minimally) angle-rigid.
\end{proposition}

\begin{proof}
    If the rank of $R(G,\cmap,p)$ is $2|V|+|\cimg| - 4$, then there is a nonzero $(2|V|+|\cimg|-4)$-minor
    determinant. As a function of the parameters $p$, this minor is a polynomial function
    not identically equal to zero. Thus, taking the ideal of all non-zero $(2|V|+|\cimg|-4)$-minor
    determinants yields a nonzero ideal defining a Zariski-closed subset of parameter space.
    The angle framework $(G,\cmap,p)$ is (minimally) infinitesimally angle-rigid for any $p$ in the complement of
    that subset.
\end{proof}

\section{A necessary condition for minimal angle-rigidity}
\label{sec:necessary}
In this section we prove necessary conditions for minimal angle-rigidity. First, using basic linear algebra and dimension counting techniques, one can derive the following Maxwell-type necessary condition. We omit the proof.

\begin{lemma}\label{lem:maxwellangle}
    Let $(G,\cmap)$ be minimally angle-rigid.
    Then for each subgraph $H \subseteq G$, the following inequality holds:
    \begin{equation*}
       |E(H)| \leq 2|V(H)|+\chi(H) - 4,
    \end{equation*}
    where $\chi(H)$ is the number of colors among the edges of $H$.
\end{lemma}

We next compare the necessary condition described in \Cref{lem:maxwellangle} to two related combinatorial statements.
Throughout the remainder of the section we fix $\mathcal{R}_2(V)$ to be the restriction of the rigidity matroid to the complete graph with vertex set $V$.

\begin{proposition}\label{prop:transversal}
      Let $(G,\cmap)$ be a colored graph with $G=(V,E)$. Consider the following three conditions:
  \begin{enumerate}
      \item\label{it:transversal:basis} There exists a set $F = \{e_1,\ldots,e_{|\cimg|}\}$ where $\cmap(e_i)= c_i$ for each
        color $c_i$, and $(E\setminus F) + e_i$ is a basis of $\mathcal{R}_2(V)$ for each color $c_i$.
      \item\label{it:transversal:cbasis} For each color $c_i$, there exists a set $F_i = \{e_j\}_{j \neq i} \subseteq E$
        where $\cmap(e_j) = c_j$ for each $j$ such that $E\setminus F_i$ is a basis of $\mathcal{R}_2(V)$.
      \item\label{it:transversal:sub} For each subgraph $H \subseteq G$, the following inequality holds: $|E(H)| \leq 2|V(H)|+\chi(H) - 4$,
        where $\chi(H)$ is the number of colors among the edges of $H$.
  \end{enumerate}
  Condition \ref{it:transversal:basis} implies \ref{it:transversal:cbasis}, which implies \ref{it:transversal:sub}.
  Condition \ref{it:transversal:sub} does not imply \ref{it:transversal:basis} or \ref{it:transversal:cbasis}, and \ref{it:transversal:cbasis} does not imply \ref{it:transversal:basis}.
\end{proposition}

\begin{proof}
    \noindent \ref{it:transversal:basis} $\Rightarrow$ \ref{it:transversal:cbasis}. For each $i$ we simply set $F_i = F \setminus e_i$. The sets $F_1,\ldots, F_{|\cimg|}$ satisfy the desired conditions by the definition of the set $F$. \\
    \noindent \ref{it:transversal:cbasis} $\Rightarrow$ \ref{it:transversal:sub}. Consider a subgraph $H \subseteq G$.
    Suppose, without loss of generality, that $c_1$ is among the colors
    on the edges of $H$.
    Let $H_1 = H \cap (E \setminus F_1)$ and $H_2 = H \cap F_1$.
    As $E \setminus F_1$ is a basis in $\mathcal{R}_2$,
    we have that $|E(H_1)| \leq 2|V(H_1)|-3$.
    Since $|E(H_2)| \leq \chi(H) - 1$,
    this implies
    \begin{equation*}
        |E(H)| = |E(H_1)| + |E(H_2)| \leq 2|V(H_1)|-3 + (\chi(H)-1) \leq 2|V(H)| + \chi(H) - 4.
    \end{equation*}

    \noindent \ref{it:transversal:sub} $\notimplies$ \ref{it:transversal:basis} or \ref{it:transversal:cbasis}. Consider the left-hand graph in \Cref{fig:counter}.
    One may check directly that condition \ref{it:transversal:sub} is satisfied. Condition \ref{it:transversal:cbasis}, on the other hand, is not
    satisfied: If \ref{it:transversal:cbasis} held then there would be a red edge that could be omitted leaving behind a basis of $\mathcal{R}_2(V)$.
    However, removing either
    red edge leaves a $K_4$ in the graph. If condition \ref{it:transversal:cbasis} fails, then condition \ref{it:transversal:basis} must also fail.\\
    
    \noindent \ref{it:transversal:cbasis} $\notimplies$ \ref{it:transversal:basis}. Consider the right-hand graph in \Cref{fig:counter}.
    Condition \ref{it:transversal:cbasis} is satisfied since $E \setminus (e_1 \cup e_2^*)$,
    $E \setminus (e_1 \cup e_3^*)$ and $E \setminus (e_2\cup e_3^*)$ are bases of $\mathcal{R}_2(V)$.
    We claim that Condition \ref{it:transversal:basis} fails. Suppose such an $F= \{x,y,z\}$ exists.
    If $F$ contains an edge not contained in one of the two copies of $K_4$, say the edge $x$, then $(E \setminus F) + y = E \setminus \{x,z\}$ leaves some $K_4$ intact,
    which means it cannot be a basis of $\mathcal{R}_2(V)$. We conclude that every edge of $F$ must be contained in some $K_4$.
    If $F$ has all three edges in the two $K_4$'s, one of the copies of $K_4$ must contain two
    edges while the other contains one, say $x$. This means that $(E \setminus F) + x$
    contains all the edges of one of the copies of $K_4$, so also fails to be a basis of $\mathcal{R}_2(V)$.
\end{proof}

\begin{figure}[ht]
    \centering
    \begin{tikzpicture}[scale=2.5]
        \node[vertex] (1) at (0,0) {};
        \node[vertex] (2) at (15:1) {};
        \node[vertex] (3) at (15+60:1) {};
        \node[vertex] (4) at ($1/3*(1)+1/3*(2)+1/3*(3)$) {};
        \node[vertex] (5) at (2.5,0) {};
        \node[vertex] (6) at ($(5)+(180-15:1)$) {};
        \node[vertex] (7) at ($(5)+(180-60-15:1)$) {};
        \node[vertex] (8) at ($1/3*(5)+1/3*(6)+1/3*(7)$) {};
        \draw[edge] (1)edge(2) (2)edge(3) (3)edge(1) (1)edge(4) (2)edge(4);
        \draw[edge] (1)edge(5) (3)edge(7);
        \draw[edge] (5)edge(6) (6)edge(7) (7)edge(5) (5)edge(8) (6)edge(8);
        \draw[edge,colR] (3)edge(4) (7)edge(8);
    \end{tikzpicture}
    \qquad
    \begin{tikzpicture}[scale=2.5]
        \node[vertex] (1) at (0,0) {};
        \node[vertex] (2) at (15:1) {};
        \node[vertex] (3) at (15+60:1) {};
        \node[vertex] (4) at ($1/3*(1)+1/3*(2)+1/3*(3)$) {};
        \node[vertex] (5) at (2.5,0) {};
        \node[vertex] (6) at ($(5)+(180-15:1)$) {};
        \node[vertex] (7) at ($(5)+(180-60-15:1)$) {};
        \node[vertex] (8) at ($1/3*(5)+1/3*(6)+1/3*(7)$) {};
        \draw[edge] (1)edge(2) (2)to node[labelsty,right]{$e_3$}(3) (3)edge(1) (2)edge(4);
        \draw[edge] (1)edge(5) (3)edge(7) (2)edge(6);
        \draw[edge] (5)edge(6) (6)to node[labelsty,left]{$e_3^*$}(7) (7)edge(5) (6)edge(8);
        \draw[edge,colR] (3)to node[labelsty,colR,left=-1pt,pos=0.9]{$e_1$}(4) (5)to node[labelsty,colR,left,pos=0.7]{$e_1^*$}(8);
        \draw[edge,colB] (1)to node[labelsty,colB,right,pos=0.65]{$e_2$}(4) (7)to node[labelsty,colB,right=-1pt,pos=0.9]{$e_2^*$}(8);
    \end{tikzpicture}
\caption{Counterexample graphs.}
\label{fig:counter}
\end{figure}

The graph on the left of \Cref{fig:counter} is not minimally angle-rigid (see \Cref{sec:extensions} for more details).
Hence, property \ref{it:transversal:sub} of \Cref{prop:transversal} (see \Cref{lem:maxwellangle}) is a necessary but not sufficient condition for minimal angle-rigidity.
The graph on the right of \Cref{fig:counter} is minimally angle-rigid but does not satisfy property \ref{it:transversal:basis} of \Cref{prop:transversal}.
Hence, property \ref{it:transversal:basis} of \Cref{prop:transversal} is not a necessary condition for minimal angle-rigidity.
We discuss the potential sufficiency of \ref{it:transversal:cbasis} of \Cref{prop:transversal} in \Cref{t:transveral}.
However, we do have positive results in some specific cases. In \Cref{sec:extensions} we use two simple extension operations to generate rigid angle frameworks and show that \ref{it:transversal:cbasis} of \Cref{prop:transversal} is sufficient (and hence provide a combinatorial characterization) when there are only 2 color classes.

We now give a sharper result showing that property \ref{it:transversal:cbasis} of \Cref{prop:transversal} is necessary. To this end let $(G,\cmap)$ be a colored graph, and let $c_j \in \cimg$ be a color.
Denote by $T_{\cmap,j}(G)$,
where the ground set is $E(G)$ and the bases are sets of $|\cimg|-1$ elements, where for each color $c_i \neq c_j$ there exists exactly one element $e_i$ where $\cmap(e_i) = c_i$. The next lemma, whose proof is simply unpacking the definitions, reformulates property \ref{it:transversal:cbasis} of \Cref{prop:transversal} in terms of transversal.

\begin{lemma}
Let $G=(V,E)$ and let $(G,\cmap)$ be a colored graph. Then
for each color $c_i$, the following are equivalent:
\begin{enumerate}
    \item there exists a set $F_i = \{e_j\}_{j \neq i} \subseteq E$
    where $\cmap(e_j) = c_j$ for each $j$ such that $E\setminus F_i$ is a basis of $\mathcal{R}_2(V)$;
    \item there exists a transversal $X \in T_{\cmap,i}(G)$ and a minimally rigid graph $H$ such that $G = H + X$ and $X \cap E(H) = \emptyset$.
\end{enumerate}
\end{lemma}

\begin{lemma} \label{l:exclude1}
    Suppose $(G,\cmap)$ is minimally angle-rigid. Then the submatrix of $R(G,\cmap,p)$ excluding the column for any color $c_j$ has rank $2|V|+|\cimg|-4$.
\end{lemma}
\begin{proof}
    Take $p$ generic. Since $(G,\cmap)$ is minimally angle-rigid, $R(G,\cmap,p)$ has $2|V|+|\cimg|-4$ rows and $2|V|+|\cimg|$ columns and has rank $2|V|+|\cimg|-4$. In particular, this implies
    that it has a 4-dimensional kernel as in \Cref{l:trivflex2}. Only the vector $(u^p,\mathbf{1})$ is supported on the color columns, and it has full support.

    Now consider $R_{\hat{i}}(G,\cmap,p)$, the $(2|V|+|\cimg|-4) \times (2|V|+|\cimg|-1)$ matrix obtained by dropping one color column.  Obviously it retains the three first kernel vectors of \Cref{l:trivflex2}, namely $(u^{(1,0)},\mathbf{0})$, $(u^{(0,1)},\mathbf{0})$, and $(u^{p^\perp},\mathbf{0})$. Let $\mathbf{v}$ be in $\ker R_{\hat{i}}(G,\cmap,p)$. The vector $(\mathbf{v},0)$ must then be in $\ker R(G,\cmap,p)$, which implies it is in the span of $(u^{(1,0)},\mathbf{0})$, $(u^{(0,1)},\mathbf{0})$, $(u^{p^\perp},\mathbf{0})$, and $(u^p,\mathbf{1})$. The zero in the final coordinate
    implies that the final vector has coefficient zero in the expansion of $(\mathbf{v},0)$, but this in turn implies that $\mathbf{v}$ is in the span of $(u^{(1,0)},\mathbf{0})$, $(u^{(0,1)},\mathbf{0})$, and $(u^{p^\perp},\mathbf{0})$. Thus the 
    kernel of $R_{\hat{i}}(G,\cmap,p)$ had dimension precisely 3. Thus its rank is
    also $2|V|+|\cimg|-4$.
\end{proof}

\begin{theorem}\label{t:transveral}
    If a colored graph $(G,\cmap)$ is minimally angle-rigid, then for each color $c_j$,
    there exists a transversal $X \in T_{\cmap,j}(G)$ and a minimally
    rigid graph $H$ such that $G = H + X$ and $X \cap E(H) = \emptyset$.
\end{theorem}

\begin{proof}
    \noindent For a generic angle framework $(G,\cmap,p)$, fix the linear map
    \begin{equation*}
        \Phi: \ker R(G,p)^T \rightarrow \RR^{c}, ~ \omega \mapsto \left( \sum_{vw \in E_{c_i}} \omega_{vw} \|p_v-p_w\|^2 \right)_{c_i \in \cimg}.\end{equation*}
    We now calculate the rank of $\Phi$.
    The graph $G$ has $2|V| + |\cimg|-4$ edges and a rigid subgraph, 
     so $R(G,p)$ has $2|V|+|C|-4$ rows and rank $2|V| -3$. Rank-nullity thus implies that $\ker R(G,p)^T$ has dimension $|\cimg| - 1$.
    Note that $\omega \in \ker \Phi$ if and only if it is an equilibrium stress of $(G,\cmap,p)$; minimal angle-rigidity of $(G,\cmap,p)$ thus implies $\Phi$ is injective.
    Therefore, $\rank \Phi = |\cimg|-1$.

    We can actually specify $\im \Phi$ explicitly. Since $\sum_{vw \in E} \omega_{vw}\|p_v-p_w\|^2 = 0$ (see, for example, \cite{connelly1982energy}),
    every point in the image of $\Phi$ is orthogonal to the all-ones vector. Since the rank is $|\cimg|-1$, we conclude that $\Phi$ surjects onto the orthogonal complement of the all-ones vector.
    In particular, there exists $\omega_i \in \ker R(G,p)^T$ such that $\Phi(\omega_i) = \mathbf{u}_i - \mathbf{u}_j$, where $\mathbf{u}_i$ is the vector with entry $i$ equal to 1 and all others equal to 0.
    Moreover, for each $i \neq j$ we can pick $e_i \in E_{c_i}$ so that $\omega_i(e_i) \neq 0$.
    With this, set $X = \{e_i :i \neq j\}$.
    
    Order the rows of $R_{\hat{j}}(G,\cmap,p)$ so that rows corresponding to $X$ are at the top with row $e_s$ above row $e_t$ if $s<t$.
    We now apply the following matrix operations to $R_{\hat{j}}(G,\cmap,p)$ for each $i \neq j$ in turn:
    (i) multiply row $e_i$ by $\omega_i(e_i)$,
    (ii) add $\omega_i(e)$ times row $e$ to row $e_i$ for each $e \in E$. These two steps
    cancel out the entries in $R(G,p)$ in rows
    corresponding to $X$, since $\omega \in \ker R(G,p)^T$. As for the color columns in the rows of $X$, our 
    choice of $\omega_i$ mapping to $\mathbf{u}_i - \mathbf{u}_j$ forces this to output 
    a 1 in the column corresponding to color $c_i$ and zero elsewhere (since color $c_j$ has been dropped). We end up with the following matrix:
    \begin{equation*}
        M=
        \begin{bmatrix}
            \mathbf{0}_{|F| \times 2|V|} & I_{c-1} \\
            R(G -X,p) & A
        \end{bmatrix}
    \end{equation*}
    for some $(|E\setminus F|) \times (c-1)$ matrix $A$.
    As $\rank R_{\hat{j}}(G,\cmap,p) = \rank R(G,\cmap,p)$ (\Cref{l:exclude1}) and $\rank M = \rank R_{\hat{j}}(G,\cmap,p)$,
    we have $\rank R(G-X,p) = \rank M - \rank I_{c-1} = 2|V|-3$.
    Thus $(G-X,p)$ is infinitesimally rigid.
\end{proof}

We conclude by setting forth the conjecture that the matroid takes the form specified by
the second property in \Cref{prop:transversal}, i.e., that the converse of \Cref{t:transveral} holds.

\begin{conjecture}
    A colored graph $(G,\cmap)$ is minimally angle-rigid if and only if for each color $c_j$,
    there exists a transversal edge set $X \in T_{\cmap,j}(G)$ and a minimally rigid graph $H$ such that $G = H \cup X$ and $X \cap E(H) = \emptyset$.
    \label{conj:structure}
\end{conjecture}

\section{Extension moves for angle-rigid graphs}
\label{sec:extensions}

Given a graph $G=(V,E)$, a \emph{0-extension} creates a new graph $G'$ which is obtained from $G$ by adding one new vertex $w$ and 2 new edges both incident to $w$. A \emph{1-extension} creates a new graph $G'$ by deleting an edge $xy$ from $E$ and adding a new vertex $w$ and 3 new edges all incident to $w$ including the edges $wx,wy$.

We generalize these operations to the colored case to form a new colored graph $(G',\cmap')$ from $(G,\cmap)$. In the 0-extension case the only constraint is that the two new edges use colors from the color set of the original graph. For the 1-extension the three new edges use colors from the color set of the original graph but we need an extra constraint.
Specifically, we say that a 1-extension is \emph{color-preserving} if either $\cmap'(wx) = \cmap(xy)$ or $\cmap'(wy) = \cmap(xy)$.

\begin{lemma}\label{lem:0ext}
    Let $(G',\cmap')$ be formed from $(G,\cmap)$ by a 0-extension.
    Then $(G,\cmap)$ is independent if and only if $(G',\cmap')$ is independent.
\end{lemma}

\begin{proof}
    Let $w$ be the new vertex in $G'$ that is adjacent to $x,y \in V$.
    Choose a generic realization $p'$ of $G'$ and define $p$ to be the realization of $G$ with $p_v = p'_v$ for all $v \in V$.
    Let $A$ be the $2\times 2$ non-singular matrix with rows $(p_w - p_x)^T$ and $(p_w - p_y)^T$ and $O$ be the $|E| \times 2$ all zeroes matrix.
    Then there exists a $2 \times (2|V| + |\cimg|)$ matrix $B$ so that

    \begin{align*}
        R(G',c',p') =
        \begin{pmatrix}
            O  &  R(G,\cmap,p) \\
            A &  B
        \end{pmatrix}.
    \end{align*}

   \noindent Hence, $\rank R(G',c',p') = \rank R(G,\cmap,p) + 2$ as required.
\end{proof}

The same holds for color-preserving 1-extensions.

\begin{lemma}\label{lem:1ext}
    Let $(G',\cmap')$ be formed from $(G,\cmap)$ by a color-preserving 1-extension.
    If $(G,\cmap)$ is independent then $(G',c')$ is independent.
\end{lemma}

\begin{proof}
    Suppose $(G',\cmap')$ is formed from $(G,\cmap)$ by a 1-extension that removes an edge $xy$,
    adds a new vertex $w$ and adds the edges $wx,wy,wz$ for some other vertex $z$.
    Further suppose that $\cmap'(wx) = \cmap(xy)$.
    Choose a generic realization $p$ of $(G,\cmap)$ and define for each $t \in \RR \setminus \{0,1\}$ the realization $p^t$ of $(G',c')$ where $p^t_v= p_v$ for all $v \in V$ and $p^t_w = tp_x + (1-t) p_y$.
    
    Define $M'_t$ to be the matrix formed from $R(G',c',p^t)$ by adding $(1-t)/t$ times row $wy$ to row $wx$.
    The only rows with non-zero entries in the $w$ columns are row $wy$ (with $p^t_w - p^t_y$) and row $wz$ (with $p^t_w - p^t_z$);
    this is because in row $wx$ and column $w$ we have
    \begin{align*}
        P^t_{wx} + ((1-t)/t)P^t_{wy} = (1-t)P_{yx}+(1-t)P_{xy} = 0
    \end{align*}
    Now define $M_t$ to be the matrix formed from $M'_t$ by deleting rows $wy$ and $wz$ and the columns corresponding to $w$, and then multiplying row $wx$ by $(1-t)^{-1}$.
    As $P^t_{wy}$ and $P^t_{wz}$ are linearly independent,
    we have $\rank M_t = \rank(G',\cmap',p^t) -2$.
    
    Suppose $\cmap'(wx) = \cmap'(wy) = \cmap(xy)$.
    The row $wx$ of $M_t$ is of the form
    \begin{align*}
        ( \quad \overbrace{P_{xy}}^{x} \quad \ldots \quad \overbrace{-P_{xy}}^{y} \quad \ldots \quad  \overbrace{-P_{xy}\cdot P_{xy}}^{\cmap(xy)} \quad ).
    \end{align*}
    Hence, $M_{t} = R(G,\cmap,p)$ for all values of $t$,
    and so $\rank R(G',\cmap',p^{1/2}) = \rank R(G,\cmap,p) +2$ as required.
    Now suppose $\cmap'(wx) \neq \cmap'(wy)$.
    The row $wx$ of $M_t$ is of the form
    \begin{align*}
        ( \quad \overbrace{(1-t)P_{xy}}^{x} \quad \ldots \quad \overbrace{-(1-t)P_{xy}}^{y} \quad \ldots \quad \overbrace{-(1-t)P_{xy}\cdot P_{xy}}^{\cmap(xy)} \quad \overbrace{-t P_{xy} \cdot P_{xy}}^{\cmap'(wy)} \:\: ).
    \end{align*}
    We now note that $\lim_{t \rightarrow 0} M_t = R(G,\cmap,p)$,
    hence for sufficiently small $t$ we get $\rank R(G',c',p^{t}) = \rank R(G,\cmap,p) +2$ as required.
\end{proof}

We next show that in the special case when $|\cimg|=2$ they are enough to derive a complete characterization. To this end we need the following basic lemma.

\begin{lemma}\label{lem:base}
    Let $(K_4,c)$ be a colored graph with $|\cimg| \geq 2$.
    Then, there exists an independent angle framework $(K_4,c,p)$.
\end{lemma}

\begin{proof}
    There are 5 non-isomorphic bichromatic graphs $(K_4,c)$, see \Cref{fig:K4col}, which we can easily check are independent by choosing random realizations.

    \begin{figure}[ht]
        \centering
        \begin{tikzpicture}[scale=1.2]
            \node[vertex] (1) at (0,0) {};
            \node[vertex] (2) at (1,0) {};
            \node[vertex] (3) at (1,1) {};
            \node[vertex] (4) at (0,1) {};
            \draw[edge,colR] (1)--(3) (1)--(4) (2)--(3) (2)--(4) (3)--(4);
            \draw[edge,colB] (1)--(2);
        \end{tikzpicture}
        \quad
        \begin{tikzpicture}[scale=1.2]
            \node[vertex] (1) at (0,0) {};
            \node[vertex] (2) at (1,0) {};
            \node[vertex] (3) at (1,1) {};
            \node[vertex] (4) at (0,1) {};
            \draw[edge,colR] (1)--(4) (2)--(3) (2)--(4) (3)--(4);
            \draw[edge,colB] (1)--(2) (1)--(3);
        \end{tikzpicture}
        \quad
        \begin{tikzpicture}[scale=1.2]
            \node[vertex] (1) at (0,0) {};
            \node[vertex] (2) at (1,0) {};
            \node[vertex] (3) at (1,1) {};
            \node[vertex] (4) at (0,1) {};
            \draw[edge,colR] (1)--(3) (1)--(4) (2)--(3) (2)--(4);
            \draw[edge,colB] (1)--(2) (3)--(4);
        \end{tikzpicture}
        \quad
        \begin{tikzpicture}[scale=1.2]
            \node[vertex] (1) at (0,0) {};
            \node[vertex] (2) at (1,0) {};
            \node[vertex] (3) at (1,1) {};
            \node[vertex] (4) at (0,1) {};
            \draw[edge,colR] (2)--(3) (2)--(4) (3)--(4);
            \draw[edge,colB] (1)--(2) (1)--(3) (1)--(4);
        \end{tikzpicture}
        \quad
        \begin{tikzpicture}[scale=1.2]
            \node[vertex] (1) at (0,0) {};
            \node[vertex] (2) at (1,0) {};
            \node[vertex] (3) at (1,1) {};
            \node[vertex] (4) at (0,1) {};
            \draw[edge,colR]  (1)--(4) (2)--(3)  (3)--(4);
            \draw[edge,colB] (1)--(2) (1)--(3) (2)--(4);
        \end{tikzpicture}
        \caption{The 5 non-isomorphic bichromatic colorings of $K_4$.}
        \label{fig:K4col}
    \end{figure}
    
    Suppose the result holds for any coloring of $K_4$ with at least 2 and at most $k \geq 2$ colors,
    and suppose $(K_4,c)$ is a colored graph with $|\cimg| = k +1$.
    Let $c'$ be the $k$-coloring of $K_4$ formed by setting every edge of color $c_{k+1}$ to be of color $c_k$.
    The result now follows as $(K_4,c')$ is independent and $\rank R(K_4,c',p) \leq \rank R(K_4,c,p)$ for any realization $p$ of $K_4$.
\end{proof}

For the next result, if $G=(V,E)$ is a graph and $E$ is a circuit in  $\mathcal{R}_2$ then we say that $G$ is an \emph{$\mathcal{R}_2$-circuit}.

\begin{theorem}\label{thm:2col}
    Let $(G,\cmap)$ be a colored graph with $|\cimg|=2$.
    Then the following are equivalent:
    \begin{enumerate}
        \item \label{thm:2col1} $(G,\cmap)$ is minimally angle-rigid.
        \item \label{thm:2col2} $|E|=2|V|-2$, and $G$ contains a unique $\mathcal{R}_2$-circuit that contains edges of both colors.
        \item \label{thm:2col3} $(G,\cmap)$ can be constructed from a bichromatic copy of $K_4$ by a sequence of 0-extensions and color-preserving 1-extensions.
    \end{enumerate}
\end{theorem}

In the proof, we need the following: Let $G=(V,E)$ be a graph and, for $X\subset V$, let $i_G(X)$ denote the number of edges in the subgraph of $G$ induced by $X$. We say that $G$ is \emph{Laman} if $i_G(X)\leq 2|X|-3$ for all $X$ with $|X|\geq 2$ and $i_G(V)=2|V|-3$. Also $X\subset V$ is \emph{critical} if $i_G(X)=2|X|-3$.
We also use $d(X,Y)$ to denote the number of edges in $G$ of the form $xy$ with $x\in X,y\in Y$ for disjoint sets $X,Y\subset V$.

Recall that \cite{pollaczekgeiringer1927,laman1970} showed that $G$ is minimally rigid (as a generic bar-joint framework) in the plane if and only if $G$ is Laman. Hence, all $\mathcal{R}_2$-circuits are rigid.
It is now immediate that condition \ref{thm:2col2} implies that $G$ is rigid as a bar-joint framework. An elementary property of such rigid graphs is that they are 2-edge-connected and any 2-edge-separation has one component of size 1 (i.\,e.\ the separation simply separates a degree 2 vertex from the rest of the graph).
The condition is also checkable efficiently by the
pebble game algorithm \cite{jacobs1997pebble}.

\begin{lemma}[\cite{pollaczekgeiringer1927,laman1970}]\label{lem:lamanreduce}
    Let $G=(V,E)$ be a Laman graph and let $v\in V$ be a vertex of degree 3 in $G$. Then there exists a pair $\{x,y\}\subset N(v)$ such that $G-v+xy$ is a Laman graph.
\end{lemma}

\noindent It is convenient to introduce the terms \emph{0-reduction} and \emph{color-preserving 1-reduction} for the colored graph operations that `undo' a 0-extension and a color preserving 1-extension respectively.

\begin{proof}[Proof of \Cref{thm:2col}]
    \ref{thm:2col1} $\Rightarrow$ \ref{thm:2col2}:
    This follows from \Cref{t:transveral}.
    
    \ref{thm:2col2} $\Rightarrow$ \ref{thm:2col3}:
    Let $\mathcal{D}$ be the set of all colored graphs that satisfy \ref{thm:2col2}.
    It suffices to show that any colored graph with at least 5 vertices in $\mathcal{D}$ can be reduced to another graph in $\mathcal{D}$ by a 0-reduction or color-preserving 1-reduction.
    Choose any colored graph $(G,\cmap) \in \mathcal{D}$ with $|V| \geq 5$.
    Since $|E|=2|V|-2$, $G$ has a vertex of degree less than 4 and since $G$ contains a unique $\mathcal{R}_2$-circuit it has no vertex of degree less than 2.
    If $(G,\cmap)$ has a vertex $v$ of degree 2 then we can apply a 0-reduction to that vertex. (Note that if $G-v$ was monochromatic then it would contain an $\mathcal{R}_2$-circuit which violates condition \ref{thm:2col2} and this circuit would also be contained in $G$.)
    Hence, we may assume that the minimum degree in $G$ is 3.

    Since $|E|=2|V|-2$ and the minimum degree is 3, $G$ contains at least four vertices of degree 3.
    Let $H$ be the unique $\mathcal{R}_2$-circuit and suppose that there is a vertex $v$ with degree 3 that is not in $H$.
    If $v$ has three neighbors in $H$, then for any $e \in H$, $(H-e)\cup v$ fails the Laman count; this would imply it
    contains a circuit $H'$ distinct from $H$, contrary to assumption. Thus, $v$ has at most two neighbors in $H$.
    Delete from $G$ an edge $e$ of $H$ (not incident to both of the two neighbors) and we have a minimally rigid graph. Hence, by \Cref{lem:lamanreduce} there is a 1-reduction on $v$ to a smaller minimally rigid graph with any color on the new edge. Moreover the vertex set of $H-e$ is a critical set so the new edge has at most one end-vertex in $H$. We can now re-add $e$ (with its color) to get the same unique $\mathcal{R}_2$-circuit.

    If all degree 3 vertices are in $H$, we may assume $H\neq K_4$ (otherwise, $G$ would be disconnected).
    Since $|V(H)|\geq 5$, we may delete $xy$ from $H-v$ such that $|\{x,y\}\cap N(v)|\leq 1$. Then $G-xy$ is a minimally rigid graph and \Cref{lem:lamanreduce} implies there exists a 1-reduction deleting $v$ in $G-xy$, and adding $ab$ (with color prescribed by the color-preserving reduction) for some $a,b\in N(v)$, that results in a smaller minimally rigid graph. Now add $xy$ back with its original color to obtain $G'$. The graph $G'$ contains a unique $\mathcal{R}_2$-circuit $H'$.
 
    It remains to show that $H'$ contains edges of both colors. If the 3 edges incident to $v$ have the same color this is trivial. Suppose $v$ is incident to two blue edges $av,bv$ and one red edge $cv$. The conclusion is also trivial if either $ac$ or $bc$ is added in the 1-reduction. Hence, we may suppose that $ab$ is added, this forces us to color $ab$ blue. The conclusion fails if and only if $ac$ was the unique red edge of $(G,\cmap)$. However, there are at least 4 vertices of degree 3 in $H$. So we may choose $u\neq v$ in $H$ of degree 3 and repeat the argument. Clearly, at the final step, there is more than one red edge in $(G,\cmap)$.
    
    \ref{thm:2col3} $\Rightarrow$ \ref{thm:2col1}:
    This follows from \Cref{lem:0ext,lem:1ext,lem:base}.
\end{proof}

We expect that it may be possible to use similar techniques (though at least one additional operation is certainly needed) to resolve the case when $|\cimg|=3$. However, in general, different techniques seem to be needed.
The 2-color case has a nice application for bar-and-joint frameworks.

\begin{corollary}
    Let $(G,p)$ be a generic framework in $\RR^2$ such that $G$ is an $\mathcal{R}_2$-circuit.
    Let $\omega$ be an equilibrium stress of $(G,p)$.
    Then, for a non-empty $F \subseteq E$,
    we have
    \begin{equation*}
        \sum_{vw \in F} \omega_{vw} \|p_v-p_w\|^2 = 0 \qquad \text{ if and only if   } \qquad  F = E.
    \end{equation*}
\end{corollary}

\begin{proof}
    Choose any proper subset $F \subset E$.
    Define $\cmap$ to be the coloring of $G$, where $\cmap(e) = c_1$ if $e \in F$ and $\cmap(e) = c_2$ if $e \in E \setminus F$.
    By \Cref{thm:2col},
    $(G,\cmap,p)$ is minimally infinitesimally angle-rigid.
    Hence, $\sum_{vw \in F} \omega_{vw} \|p_v-p_w\|^2 \neq 0$ as required.
\end{proof}

\subsection{Computational results}

\Cref{thm:2col} gives a tool to construct all 2-colored graphs $(G,\cmap)$ that are minimally angle-rigid.
Our other tool to test a given $(G,\cmap)$ for angle-rigidity is to apply \Cref{prop:anglerigid} to a random realization $p$. In the case when $(G,\cmap,p)$ is infinitesimally angle-rigid we know $(G,\cmap)$ is also angle-rigid. On the other hand, if $(G,\cmap,p)$ is not infinitesimally angle-rigid, $(G,\cmap)$ can only be assumed to be non-angle-rigid, if we assume our random realization was sufficiently generic. Using these two tools we can analyze the angle-rigid colored graphs computationally.

Let $G$ be a graph with $|E|=2|V|-2$. We say that $G$ is \emph{2-color-rigid} if there is a color map $\cmap$ with $|\cimg|=2$ such that $(G,\cmap)$ is minimally angle-rigid.
For $|V|=5$ we have two 2-color-rigid graphs ($G_1$ and $G_2$ from \Cref{fig:v5}). $G_1$ admits 45 different 2-color maps (up to isomorphism) such that $(G_1,\cmap)$ is angle-rigid.
$G_2$ has 26 such maps. Hence, in total there are 71 non-isomorphic colored graphs\footnote{Here we say two colored graphs $(G,\cmap)$ and $(G',c')$ are isomorphic if there exists a graph isomorphism $\phi:G \rightarrow G'$ so that $c = c' \circ \phi$.} 
$(G,\cmap)$ with $|\cimg|=2$ that  are minimally angle-rigid. See \Cref{tab:counting} for further data.
\begin{table}[ht]
    \centering
    \begin{tabular}{rrrr}
        \toprule
         $|V|$ & graphs & 2-color-rigid & 2-colored angle-rigid \\\midrule
         4 & 1 & 1 & 5\\
         5 & 2 & 2 & 71\\
         6 & 12 & 12 & 2227\\
         7 & 97 & 91 & 99148\\
         8 & 1113 & 1003 \\
         9 & 17117 & 14870 \\
         \bottomrule
    \end{tabular}
    \caption{Graphs with $|E|=2|V|-2$ and minimum degree 2 in the second column and the number of 2-color-rigid graphs in column three. Last column: number of non-isomorphic colored graphs $(G,\cmap)$ with $|\cimg|=2$ which are minimally angle-rigid.}
    \label{tab:counting}
\end{table}

\begin{figure}[ht]
    \centering
    \begin{tikzpicture}[scale=1.2]
        \node[vertex] (1) at (0,0) {};
        \node[vertex] (2) at (1,0) {};
        \node[vertex] (3) at (1,1) {};
        \node[vertex] (4) at (0,1) {};
        \node[vertex] (5) at (-0.5,0.5) {};
        \draw[edge] (2)--(3) (2)--(4) (3)--(4);
        \draw[edge] (1)--(2) (1)--(3) (1)--(4);
        \draw[edge] (1)--(5) (4)--(5);
        \node[labelsty] at (0.5,-0.4) {$G_1$};
    \end{tikzpicture}
    \qquad
    \begin{tikzpicture}[scale=1.2]
        \node[vertex] (1) at (0,0) {};
        \node[vertex] (2) at (1,0) {};
        \node[vertex] (3) at (1,1) {};
        \node[vertex] (4) at (0,1) {};
        \node[vertex] (5) at (0.5,0.6) {};
        \draw[edge] (2)--(3) (3)--(4);
        \draw[edge] (1)--(2) (1)--(4);
        \draw[edge] (1)--(5) (2)--(5) (3)--(5) (4)--(5);
        \node[labelsty] at (0.5,-0.4) {$G_2$};
    \end{tikzpicture}
    \caption{Two graphs with 5 vertices that are 2-color-rigid.}
    \label{fig:v5}
\end{figure}

We have seen in \Cref{fig:K4col} that $K_4$ has 5 possible color maps in two colors which give a rigid structure. We have computed the number of such color maps for all graphs up to 7 vertices. There is for instance a graph with only 7 vertices that has more than 2000 possible color maps.
In \Cref{tab:colors} we show the minimum and maximum number of color maps obtained by graphs with less than 8 vertices.

Similarly to 2-color-rigid graphs we can take more general \emph{$k$-color-rigid} graphs which have $|E|=2|V|+k-4$ and a $k$-color map $\cmap$ that gives a minimally angle-rigid structure.
Note that, computations are done using \Cref{prop:anglerigid}, hence, they use random realizations and might therefore yield false negatives.
\Cref{tab:kcolor} shows how many $k$-color-rigid graphs there are with few vertices.

\begin{table}[ht]
    \raggedright
    \begin{minipage}[t]{0.45\textwidth}
        \centering
        \begin{tabular}{rrr}
            \toprule
             $|V|$ & minimum      & maximum \\
                   & 2-color maps & 2-color maps\\\midrule
             4 & 5 & 5 \\
             5 & 26 & 45 \\
             6 & 67 & 304 \\
             7 & 46 & 2047 \\
             \bottomrule
        \end{tabular}
        \captionof{table}{For all 2-color-rigid graphs we count the number of possible 2-color maps. The minimum and maximum of these numbers for a given $|V|$ is shown in the table.}
        \label{tab:colors}
    \end{minipage} \hspace{0.5cm} 
    \begin{minipage}[t]{0.45\textwidth}
        \centering
        \begin{tabular}{rrrrrr}
             \toprule
             $|V|$ & 2-color-rigid & 3-color-rigid & 4-color-rigid \\\midrule
             4     &           1 &           - &           - \\
             5     &           2 &           1 &           1 \\
             6     &          12 &           8 &           5 \\
             7     &          91 &          80 &          59 \\
             8     &        1003 &        1168 &             \\
             9     &       14870 & \\
             \bottomrule
        \end{tabular}
        \captionof{table}{Number of $k$-color-rigid graphs with $k\leq 6$ and less than 10 vertices.}
        \label{tab:kcolor}
    \end{minipage}
\end{table}

\section{The Algebraic Angle-Rigidity Matroid}
\label{sec:algmatroid}

One important toolbox frequently employed in rigidity theory is the theory of matroids. 
The ``independent sets'' in the classical rigidity setting can be defined using linear independence among
rows of the rigidity matrix, yielding a \emph{linear matroid}. Independence can also be taken as algebraic independence of the 
distances $\{d_{vw} = X_{vw}^2 + Y_{vw}^2\}$ in the field extension $\CC(x_v,y_v)$; this defines an
\emph{algebraic matroid} (see \cite{rosen2020algebraic} for an elementary introduction).
Since our angle rigidity matrix indexes its rows by edges as opposed to angles, the linear matroid is not immediately 
adaptable. Instead, we construct an algebraic matroid.

Let $G = (V,E)$ be a graph with an angle index set $A$ and let $e,f \in E$. 
Our previously defined angle map $\theta_A: R^G \rightarrow \RR^{A}$ has coordinates
$ \left(\arccos \left( \frac{(p_a-p_b)\cdot (p_c-p_d)}{\|p_a-p_b\| \|p_c-p_d\|} \right) \right)_{\{ab,cd\} \in A}$ which are related only via transcendental functions.
In this section, we construct the map $\theta_{ef} \to \exp(2\ci \:\theta_{ef})$, 
obtaining a set of complex numbers which are related algebraically
precisely when the original angles satisfy a geometric constraint.  
This allows us to define the \emph{algebraic angle-rigidity matroid} in terms of algebraic relations.

\begin{definition}
    Let $G = (V,E)$ be a graph, with $e,f \in E$. 
    The \emph{algebraic angle-rigidity matroid} $\mathcal{A}(G)$ (over $\fieldk$) is the algebraic matroid on the elements
    $\alpha_{ef} = X_{e}Y_{f}X_{f}^{-1}Y_{e}^{-1}$
    in the function field $\fieldk(x_1,\ldots,x_n,$ $y_1,\ldots,y_n)$. For $A$ an angle index set of $G$, let
    $\alpha_A =\{\alpha_{ef}:\{e,f\} \in A\}$.
\end{definition}

The justification of this formula for angles is as follows:
Suppose $p_1 = (\tilde{x}_1,\tilde{y}_1),\ldots, p_4 = (\tilde{x}_4,\tilde{y}_4)$ and we
want to measure the angle between line segment $L_{12}$ joining $p_1$ and $p_2$ and $L_{34}$ joining $p_3$ and $p_4$. We perform the following geometric operations:

\paragraph{Translate Segments to the Origin} Then, both $L_{12}$ and $L_{34}$ have an endpoint at the origin. After this $p_1' = (\tilde{X}_{12}, \tilde{Y}_{12})$ and $p_3' = (\tilde{X}_{34},\tilde{Y}_{34})$, while $p_2' = p_4' = 0$.
    
\paragraph{Rescale to the Unit Circle} We apply a complex change of coordinates
    inspired by \cite{capco2018counting}. Send $(\tilde{x}_k,\tilde{y}_k) \mapsto (\tilde{x}_k+ \ci\tilde{y}_k,\tilde{x} - \ci\tilde{y}_k)$ a conjugate pair of complex numbers.
    Rename these coordinates $x_k = \tilde{x}_k + \ci\tilde{y}$ and $ y_k = \tilde{x}_k - \ci\tilde{y}_k$.
    This can be extended to an invertible linear map on $\CC^2$, so preserves the underlying geometry.  Applying this transformation, $p_1'$ and $p_3'$ become
    \[    p_1''  = (\tilde{X}_{12} + \ci \tilde{Y}_{12},\quad \tilde{X}_{12} - \ci \tilde{Y}_{12}) \hspace{18mm}
        p_3'' = (\tilde{X}_{34} + \ci \tilde{Y}_{34}, \tilde{X}_{34} - \ci \tilde{Y}_{34}) \]
        or in the new variable names:
    $p_1'' = (X_{12},Y_{12})$  and $ p_3'' = (X_{34},Y_{34})$.
    Convert these coordinates to polar form, keeping in mind that the pairs are complex conjugates:
    \begin{align*}
        p_1'' &= (X_{12},Y_{12})=(r_{12}e^{\ci\theta_{12}},  r_{12}e^{-\ci\theta_{12}}) & \hspace{1.5cm}
        p_3'' &=  (X_{34}, Y_{34})  =  (r_{34}e^{\ci\theta_{34}}, r_{34}e^{-\ci\theta_{34}})
    \end{align*}
    To drop the magnitude $r_{kl}$, we can divide the two coordinates. In particular:
    \[
        p_1''' = X_{12}/Y_{12}= e^{2\ci\theta_{12}},  \hspace{2.1cm}
        p_3''' = X_{34}/Y_{34} = e^{2\ci\theta_{34}}.
    \]
    The unfortunate side effect is that $\theta_{kl}$ gets doubled, but this still specifies $\theta_{kl}$ up to $\pi$. Now both line segments (with angle doubled) terminate at the unit circle.
    
\paragraph{Rotate the Unit Circle.}
    We divide everything by $e^{2\ci\theta_{34}} = X_{34}/Y_{34}$ so that $p_3'''$ rotates to 1
    and $p_1'''$ rotates to  
    \[
        e^{2\ci(\theta_{12}-\theta_{34})} = X_{12}Y_{34}X_{34}^{-1}Y_{12}^{-1}
    \]
    Thus, we define the angle variable $\alpha_{(12)(34)}$ as the
    rational function on the right. Note that the traditional angle value in $(0,2\pi)$ can be computed via $\frac{1}{2\ci} \ln a_{(12)(34)}$, using the principal branch of the complex logarithm as claimed above. The field elements themselves are compactly 
    summarized as a Laurent monomial function of linear forms.

\begin{remark}\label{rem:martin}
    The angles can also be defined in terms of the slope matroid defined by \cite{martin2003}.
    Here Martin defines, for each $e = (uv)$, the equation
    $ m_{e} =Y_{e}X_{e}^{-1} = Y_{uv}X_{uv}^{-1}$.
    In Martin's work, no complex coordinate change
    is used; instead, he defines slopes in the traditional geometric sense.
    Still, using that formula, we have:
    $ a_{ef} = m_{f}/m_{e}$,
    where $m_{e}$ and $m_{f}$ are elements in the function field $\fieldk(m_{e})/\mathcal{S}(G)$,
    where $\mathcal{S}(G)$ is the ideal of the \emph{slope variety} defined in \cite{martin2003}.
\end{remark}

In this definition, we do not use the fact that there is a difference between $\alpha_{ef}$ and $\alpha_{fe}$. This ambiguity is justified by
\ref{it:tree:inv} in the next proposition, which observes that the field extension generated by the set of angles is invariant under shuffling the edges in an angle index.

\begin{proposition}\label{prop:tree}
    Let $A$ be an angle index set.  The following properties are satisfied by $\alpha_A$:
    \begin{enumerate}
        \item\label{it:tree:inv} $\alpha_{ef} = (\alpha_{fe})^{-1}$.
        \item\label{it:tree:cycle} If $\tilde{G}(A)$ has a cycle, then $\alpha_A$ satisfies a nontrivial polynomial relation.
        \item\label{it:tree:tree} If $\tilde{G}(A)$ is a tree and $\alpha_A$ satisfies a nontrivial polynomial relation,
            then so does every $A'$ such that $\tilde{G}(A')$ is connected on the 
            same support.
        \item\label{it:tree:forest} Suppose $\tilde{G}(A)$ is a forest with connected components $T_1,T_2,\dots, T_k$
            and $\alpha_A$ satisfies a nontrivial polynomial relation.
            If $A'$ is another angle index set for which $\tilde{G}(A')$ is a forest, and its connected components $T_1',\ldots,T_k'$
            satsify $E(T_i) = E(T_i')$, then $\alpha_{A'}$ also satisfies a nontrivial polynomial relation.
    \end{enumerate}
\end{proposition}

\begin{proof}
    \ref{it:tree:inv} falls directly out of the formula.

    For \ref{it:tree:cycle}, take the edge sequence of the cycle $e_1,e_2,\ldots
    e_k$. Observe that multiplying all the $\alpha_{e_1e_2},
    \ldots \alpha_{e_ke_1}$ results in every factor
    $X_{e_i}$ and $Y_{e_i}$ appearing exactly
    once each in numerator and denominator; hence,
    $\alpha_{e_1e_2} \alpha_{e_2e_3}\cdots \alpha_{e_ke_1} = 1.$
  
    For \ref{it:tree:tree}, take any angle index ${ef} \in A \setminus A'$.
    As the graph $\tilde{G}(A') + ef$ contains a cycle with $ef$,
    we can apply \ref{it:tree:cycle} to this cycle to write $\alpha_{ef}$ in terms of angles of $A'$,
    here using the fact that the product of a cycle equals 1.
    We now note that if $P$ is a non-trivial polynomial with $P(\alpha_A) = 0$,
    then replacing each of those in the nontrivial relation $P(\alpha_A) = 0$ yields a nontrivial rational function $q(\alpha_{A'}) = 0$, and clearing
    denominators yields the result.
    This method also directly implies \ref{it:tree:forest}.
\end{proof}

Part \ref{it:tree:cycle} of \Cref{prop:tree} implies that algebraically independent sets $\alpha_A$ must be \emph{acyclic},
(i.e. $\tilde{G}(A)$ has no cycles).
In addition,
part \ref{it:tree:tree} and \ref{it:tree:forest} of \Cref{prop:tree} leads to the same observation made in
Section 2, reframing the problem in terms of graphs with edge colorings.

\begin{corollary}\label{cor:algmat}
    The algebraic (in)dependence of an acyclic set of angles $A$ is determined by the graph $G$ and the partition into colors $c$ based on the connected components of $\tilde{G}(A)$.
\end{corollary}

It follows from \Cref{cor:algmat} that every set of angles $A$ generates a corresponding colored graph $(G,\cmap)$ by labeling the components of $\tilde{G}(A)$ as $C_1,\ldots,C_k$ and fixing $\cmap$ to be the map where $\cmap(e) = c_i$ if and only if $e \in C_i$.
This observation allows us to freely change our choice of angles $A$ to another set $A'$ so long as the connected components of $\tilde{G}(A)$ and $\tilde{G}(A')$ share the same vertices.
This can simplify computations, as is the case in the next lemma.

\begin{lemma}\label{lemma:star}
    Suppose $\tilde{G}(A)$ a tree, and $f \in E(A)$.
    Take $m_{e} = X_{e}/Y_{e}$ to be the variable defined in \Cref{rem:martin}.
    Then, $\fieldk(\alpha_A \cup \{m_{f}\}) = \fieldk(m_e:e \in E(A))$.
\end{lemma}

\begin{proof}
    By \Cref{prop:tree}, $A$ can be taken so that $\tilde{G}(A)$ is a star with central vertex $f$.
    Every $\alpha_{ef} = m_{e}/m_{f}$, so for every $e \in E(A)$, $\alpha_{ef}m_{f} = m_{e}$.
\end{proof}

With these basic properties established, we now demonstrate
that algebraic independence among $\alpha_A$ corresponds precisely to minimally angle-rigid graphs.

\begin{theorem}\label{thm:algmat}
    An angle index set $A$ with $\tilde{G}(A)$ acyclic defines a basis in the algebraic matroid $\mathcal{A}(G)$ if and only if the corresponding colored graph $(G,\cmap)$ is minimally angle-rigid.
\end{theorem}

\begin{proof}
    For convenience, we denote squared distance $S_{e} = X_{e}Y_{e}$. A set of elements in a field extension $\mathbb{F}/\fieldk$ of characteristic $0$ is algebraically independent over $\mathbb{F}$ if and only if the corresponding set of differentials is linearly independent in $\Omega_{\mathbb{F}/\fieldk}$ \cite[Theorem 16.14]{eisenbud2013}.
    We set $\mathbb{F} = \fieldk(x_v,y_v: v \in V)$ with all $x,y$ transcendental over $\fieldk$; note that $\mathbb{F}$ contains $X_{uv} = x_u - x_v$, $Y_{uv}, S_{uv}$, etc. Let $e = st$ and $f = uv$. Differentiate the defining equations of $\alpha_{ef}$ and $m_{e}$ to obtain:
    \begin{align*}
        \mathrm{d}\alpha_{ef} & = \mathrm{d}\left(m_{e}m_{f}^{-1} \right) = m_{f}^{-1}\mathrm{d}m_{e} - m_{e}m_{f}^{-2}\mathrm{d}m_{f} = a_{ef} \left(m_{e}^{-1} \mathrm{d} m_{e} - m_{kl}^{-1}\mathrm{d}m_{kl}\right) \\
        \mathrm{d}m_{e} & = \mathrm{d}\left(X_{e}Y_{e}^{-1}\right) =\dfrac{ \mathrm{d} x_{s}- \mathrm{d} x_{t}}{Y_{e}} - \dfrac{X_{e}(\mathrm{d}y_s - \mathrm{d}y_t)}{Y_{e}^2} = \dfrac{1}{Y_{e}^2}\bigg[ Y_{e}(\mathrm{d} x_{s}- \mathrm{d} x_{s}) - X_{e}(\mathrm{d}y_s - \mathrm{d}y_t)\bigg].
    \end{align*}
    Without affecting linear independence, we may rescale each $\mathrm{d}\alpha_{ef}$ and instead consider the vectors:
    \begin{align*}
        v_{ef} & := \dfrac{S_{e}S_{f}}{\alpha_{ef}}\mathrm{d}\alpha_{ef} =
        S_{e}S_{f}\left(m_{e}^{-1} \mathrm{d} m_{e} - m_{f}^{-1}\mathrm{d}m_{f}\right) = S_{f} Y_{e}^2 \mathrm{d} m_{e} - S_{e}Y_{f}^2\mathrm{d}m_{f} \\
        &~ =  S_{f} \big[ Y_{e}(\mathrm{d} x_{s}- \mathrm{d} x_{t}) - X_{e}(\mathrm{d}y_s - \mathrm{d}y_t)\big] - S_{e}\big[ Y_{f}(\mathrm{d} x_{u}- \mathrm{d} x_{v}) - X_{f}(\mathrm{d}y_u - \mathrm{d}y_v)\big].
    \end{align*}
    The $x,y$ variables are taken to be algebraically independent, so their differentials $\mathrm{d}x_v, \mathrm{d}y_v$ are
    linearly independent. Thus we may take these as a basis of a field extension $\mathbb{F}/\fieldk$, and encode the vectors $v_{ef} \in \Omega_{\mathbb{F}/\fieldk}$ as rows of a $(2|V|-4)\times 2|V|$ matrix $M(A)$.
    With this in mind, $M(A)$ is of the form below, where all unspecified entries are zero:
    \begin{small}
        \[
            \bordermatrix{ & \mathrm{d}x_s & \mathrm{d}y_s & \cdots &\mathrm{d} x_t & \mathrm{d}y_t & \cdots& \mathrm{d} x_u & \mathrm{d}y_u &\cdots & \mathrm{d}x_v & \mathrm{d}y_v  \cr
            \hspace{2mm} \vdots & \cr
            v_{ef}& S_{f} Y_{e} & -S_{f} X_{e} & & -S_{f}Y_{e} & S_{f}X_{e} & & -S_{e}Y_{f} & S_{e} X_{f} & & S_{e}Y_{f} & -S_{e}X_{f} \cr
            \hspace{2mm} \vdots & }.
        \]
    \end{small}
    
    \noindent Noting that the coefficients of the entries in $M(A)$ are all rational, we conclude that the $\mathrm{d}\alpha_{ef}$ are linearly independent (and hence $\alpha_A$ is algebraically independent) if and only if $M(A)$ has full rank for any choice of $x_v,y_v$ algebraically independent over $\mathbb{Q}$.

    Now fix a set $\{x_v,y_v\}_{v \in V} \subset \RR$, algebraically independent over $\mathbb{Q}$,
    and take $p:V \rightarrow \RR^2$ to be the realization where $p_i = (x_i,y_i)$.
    The matrix $M(A)$ can be factored as $M(A) = T_{\widetilde{G}(A)} \cdot R(G,p^\perp)$,
    where $R(G,p^\perp)$ is the rigidity matrix of the rotated framework $(G,p^\perp)$ and $T_{\widetilde{G}(A)}$ is a $(2|V|-4) \times (2|V|+|\cimg|-4)$ block matrix described as follows. Each block corresponds to a star in $\widetilde{G}(A)$. If the star has leaves $\ell_1,\ldots,\ell_r$ and center $f$, then the corresponding block is:
    \begin{equation*}
        \bordermatrix{
            & m_{\ell_1} & \cdots & m_{\ell_r} & m_f \cr
            \alpha_{\ell_1,f} & -S_{f} & & & S_{\ell_1} \cr
            \: \: \: \vdots & & \ddots & & \vdots \cr
            \alpha_{\ell_r,f} & & & -S_f & S_{\ell_r}
        }.
    \end{equation*}
    Since $(G,\cmap)$ was formed from $A$,
    every edge of $G$ is in the subscript of some angle in $A$,
    and so the matrix $T_{\widetilde{G}(A)}$ has rank $2|V|-4$.
    
    We now form one more matrix.
    Fix $J$ to be the $(2|V| +|\cimg|-4) \times (2|V|-4)$ matrix formed from $R(G,c,p^\perp)$ by replacing each entry $-(X_{e}^2 +Y_{e}^2)$ in the color column for the edge $e$ with $S_{e}$.
    We observe that
    \begin{equation*}
        T_{\widetilde{G}(A)} J = \left(T_{\widetilde{G}(A)} R(G,p^\perp) \quad \mathbf{0} \right) = \left(M(A) \quad \mathbf{0} \right).
    \end{equation*}
    Hence, the left kernels of $M(A)$ and $T_{\widetilde{G}(A)} J$ are the same.
    
    We claim that the left nullity of $M(A)$ and $J$ are equal.
    As $T_{\widetilde{G}(A)}$ has rank $2|V|-4$,
    the left nullity of $M(A)$ is at most the left nullity of $J$.
    Now choose $\omega$ in the left kernel of $J$.
    For each color $i$ we have $\sum_{\cmap(e)=i} \omega_e S_e = 0$.
    This is equivalent to the observation that for every star in $\widetilde{G}(A)$ with star center $f$ and leaves $\ell_1,\ldots,\ell_r$,
    we have
    \begin{equation}\label{eq:star}
        \omega_f= -\sum_{i=1}^r \frac{\omega_{\ell_i}}{S_f} S_{\ell_i}.
    \end{equation}
    Define the element $\nu \in \fieldk^A$ where $\nu(a_{\ell_i,f}) = \omega_{\ell_i}/S_f$ for every angle $a_{\ell_i,f}$ where $f$ is a star center in $\widetilde{G}(A)$.
    It now follows from \Cref{eq:star} that $\nu^T T_{\widetilde{G}(A)} = \omega$,
    hence
    \begin{equation*}
        \nu^T T_{\widetilde{G}(A)} J = \omega^T J = 0.
    \end{equation*}
    Since no element of the left kernel can be supported only on the star centers of $\widetilde{G}(A)$, the map $\omega \mapsto \nu$ is injective and linear. Therefore,
    it follows that the left nullity of $M(A)$ is at most the left nullity of $J$.
    Hence, the left nullities of $M(A)$ and $J$ are equal.
    
    By manipulating rows and using the linear transform
        $x = \tilde{x} + \ci\tilde{y}$,  $y = \tilde{x} - \ci\tilde{y}$,
    it is simple to show that $J$ has the same rank as $R(G,\cmap,p)$.
    Hence, $M(A)$ has full rank if and only if $R(G,\cmap,p)$ has full rank and the result follows from \Cref{thm:lineangdirection}.
\end{proof}

\addcontentsline{toc}{section}{Acknowledgments}
\section*{Acknowledgments}

This project originated from the Fields Institute
Thematic Program on Geometric Constraint Systems, Framework Rigidity, and Distance Geometry and benefited from time at the Fields Institute
Focus Program on Geometric Constraint Systems.
The authors are grateful to the Fields Institute for their hospitality and financial support.
S.\,D.\ was supported by the Heilbronn Institute for Mathematical Research.
G.\,G.\ was partially supported by the Austrian Science Fund (FWF): 10.55776/P31888.
A.\,N.\ was partially supported by EPSRC grant EP/X036723/1.

\bibliographystyle{alphaurl}
\bibliography{lib}

\newcommand{\etalchar}[1]{$^{#1}$}
\begin{thebibliography}{EWM{\etalchar{+}}03}

\bibitem[CCL21]{chen2020angle}
Liangming Chen, Ming Cao, and Chuanjiang Li.
\newblock Angle rigidity and its usage to stabilize multiagent formations in
  2-d.
\newblock {\em IEEE Transactions on Automatic Control}, 66(8):3667--3681, 2021.
\newblock \href {https://doi.org/10.1109/TAC.2020.3025539}
  {\path{doi:10.1109/TAC.2020.3025539}}.

\bibitem[CGG{\etalchar{+}}18]{capco2018counting}
Jose Capco, Matteo Gallet, Georg Grasegger, Christoph Koutschan, Niels Lubbes,
  and Josef Schicho.
\newblock The number of realizations of a {Laman} graph.
\newblock {\em SIAM Journal on Applied Algebra and Geometry}, 2(1):94--125,
  2018.
\newblock \href {https://doi.org/10.1137/17M1118312}
  {\path{doi:10.1137/17M1118312}}.

\bibitem[Con82]{connelly1982energy}
Robert Connelly.
\newblock Rigidity and energy.
\newblock {\em Inventiones Mathematicae}, 66:11--33, 1982.
\newblock \href {https://doi.org/10.1007/BF01404753}
  {\path{doi:10.1007/BF01404753}}.

\bibitem[Eis13]{eisenbud2013}
David Eisenbud.
\newblock {\em Commutative algebra: with a view toward algebraic geometry},
  volume 150.
\newblock Springer Science \& Business Media, 2013.
\newblock \href {https://doi.org/10.1007/978-1-4612-5350-1}
  {\path{doi:10.1007/978-1-4612-5350-1}}.

\bibitem[EWM{\etalchar{+}}03]{eren2003sensor}
Tolga Eren, Walter Whiteley, A.~Stephen Morse, Peter~N. Belhumeur, and
  Brian~{D.\,O.} Anderson.
\newblock Sensor and network topologies of formations with direction, bearing,
  and angle information between agents.
\newblock In {\em 42nd IEEE International Conference on Decision and Control
  (IEEE Cat. No. 03CH37475)}, volume~3, pages 3064--3069. IEEE, 2003.
\newblock \href {https://doi.org/10.1109/CDC.2003.1273093}
  {\path{doi:10.1109/CDC.2003.1273093}}.

\bibitem[GSS93]{gss1993}
Jack Graver, Brigitte Servatius, and Herman Servatius.
\newblock {\em {Combinatorial rigidity}}.
\newblock {American Mathematical Society}, Providence, RI, 1993.
\newblock \href {https://doi.org/10.1090/gsm/002} {\path{doi:10.1090/gsm/002}}.

\bibitem[HLSS{\etalchar{+}}12]{haller2012}
Kirk Haller, Audrey Lee-St.John, Meera Sitharam, Ileana Streinu, and Neil
  White.
\newblock Body-and-cad geometric constraint systems.
\newblock {\em Computational Geometry}, 45(8):385--405, 2012.
\newblock Geometric Constraints and Reasoning.
\newblock \href {https://doi.org/10.1016/j.comgeo.2010.06.003}
  {\path{doi:10.1016/j.comgeo.2010.06.003}}.

\bibitem[JH97]{jacobs1997pebble}
Donald~J. Jacobs and Bruce Hendrickson.
\newblock An algorithm for two-dimensional rigidity percolation: The pebble
  game.
\newblock {\em Journal of Computational Physics}, 137(2):346--365, 1997.
\newblock \href {https://doi.org/10.1006/jcph.1997.5809}
  {\path{doi:10.1006/jcph.1997.5809}}.

\bibitem[JO16]{jackson2016characterisation}
Bill Jackson and John Owen.
\newblock A characterisation of the generic rigidity of 2-dimensional
  point--line frameworks.
\newblock {\em Journal of Combinatorial Theory, Series B}, 119:96--121, 2016.
\newblock \href {https://doi.org/10.1016/j.jctb.2015.12.007}
  {\path{doi:10.1016/j.jctb.2015.12.007}}.

\bibitem[JZLW19]{jing2019angle}
Gangshan Jing, Guofeng Zhang, Heung Wing~Joseph Lee, and Long Wang.
\newblock Angle-based shape determination theory of planar graphs with
  application to formation stabilization.
\newblock {\em Automatica}, 105:117--129, 2019.
\newblock \href {https://doi.org/10.1016/j.automatica.2019.03.026}
  {\path{doi:10.1016/j.automatica.2019.03.026}}.

\bibitem[Lam70]{laman1970}
Gerard Laman.
\newblock On graphs and rigidity of plane skeletal structures.
\newblock {\em Journal of Engineering mathematics}, 4(4):331--340, 1970.
\newblock \href {https://doi.org/10.1007/BF01534980}
  {\path{doi:10.1007/BF01534980}}.

\bibitem[Mar03]{martin2003}
Jeremy Martin.
\newblock Geometry of graph varieties.
\newblock {\em Transactions of the American Mathematical Society},
  355(10):4151--4169, 2003.
\newblock \href {https://doi.org/10.1090/S0002-9947-03-03321-X}
  {\path{doi:10.1090/S0002-9947-03-03321-X}}.

\bibitem[PG27]{pollaczekgeiringer1927}
Hilda Pollaczek-Geiringer.
\newblock {\"Uber die Gliederung ebener Fachwerke}.
\newblock {\em ZAMM-Journal of Applied Mathematics and Mechanics/Zeitschrift
  f{\"u}r Angewandte Mathematik und Mechanik}, 7(1):58--72, 1927.
\newblock \href {https://doi.org/10.1002/zamm.19270070107}
  {\path{doi:10.1002/zamm.19270070107}}.

\bibitem[RST20]{rosen2020algebraic}
Zvi Rosen, Jessica Sidman, and Louis Theran.
\newblock Algebraic matroids in action.
\newblock {\em The American Mathematical Monthly}, 127(3):199--216, 2020.
\newblock \href {https://doi.org/10.1080/00029890.2020.1689781}
  {\path{doi:10.1080/00029890.2020.1689781}}.

\bibitem[SST22]{schulze2022coordinated}
Bernd Schulze, Hattie Serocold, and Louis Theran.
\newblock Frameworks with coordinated edge motions.
\newblock {\em SIAM Journal on Discrete Mathematics}, 36(4):2602--2618, 2022.
\newblock \href {https://doi.org/10.1137/20M1377539}
  {\path{doi:10.1137/20M1377539}}.

\bibitem[SW99]{servatius1999constraining}
Brigitte Servatius and Walter Whiteley.
\newblock Constraining plane configurations in computer-aided design:
  Combinatorics of directions and lengths.
\newblock {\em SIAM Journal on Discrete Mathematics}, 12(1):136--153, 1999.
\newblock \href {https://doi.org/10.1137/S0895480196307342}
  {\path{doi:10.1137/S0895480196307342}}.

\bibitem[SW04]{saliola2004constraining}
Franco Saliola and Walter Whiteley.
\newblock Constraining plane configurations in {CAD}: circles, lines, and
  angles in the plane.
\newblock {\em SIAM Journal on Discrete Mathematics}, 18(2):246--271, 2004.
\newblock \href {https://doi.org/10.1137/S0895480100374138}
  {\path{doi:10.1137/S0895480100374138}}.

\bibitem[Whi87]{whiteley1987parallel}
Walter Whiteley.
\newblock Parallel redrawings, 1987.
\newblock Preprint.
\newblock \href {https://doi.org/10.13140/RG.2.2.13701.91365}
  {\path{doi:10.13140/RG.2.2.13701.91365}}.

\bibitem[Whi96]{whiteley1996}
Walter Whiteley.
\newblock Some matroids from discrete applied geometry.
\newblock In {J.\,E.} Bonin, {J.\,G.} Oxley, and B.~Servatius, editors, {\em
  {Matroid theory}}, volume 197 of {\em Contemporary Mathematics}, pages
  171--311. American Mathematical Society, Providence, RI, 1996.
\newblock \href {https://doi.org/10.1090/conm/197/02540}
  {\path{doi:10.1090/conm/197/02540}}.

\bibitem[Zho06]{zhou2006combinatorial}
Yong Zhou.
\newblock {\em Combinatorial decomposition, generic independence and algebraic
  complexity of geometric constraints systems: applications in biology and
  engineering}.
\newblock PhD thesis, University of Florida, 2006.

\bibitem[ZZ15]{zhao2015bearing}
Shiyu Zhao and Daniel Zelazo.
\newblock Bearing rigidity and almost global bearing-only formation
  stabilization.
\newblock {\em IEEE Transactions on Automatic Control}, 61(5):1255--1268, 2015.
\newblock \href {https://doi.org/10.1109/TAC.2015.2459191}
  {\path{doi:10.1109/TAC.2015.2459191}}.

\end{thebibliography}

\end{document}